\documentclass[article]{elsarticle}

\usepackage{color}
\usepackage{amsmath,amsfonts,amsthm,amssymb,bm,mathtools} %
\usepackage{graphicx,psfrag,subfigure}                    %
\usepackage{enumitem}
\usepackage{epstopdf}
\usepackage{latexsym}
\usepackage{mathrsfs}                           %
\usepackage{fullpage}

\allowdisplaybreaks

\renewcommand{\d}{\mathrm{d}}
\newcommand{\e}{\mathrm{e}}
\newcommand{\R}{\mathbb{R}}

\newcommand{\bx}{\bm{x}}

\newcommand{\hL}{\mathcal{L}}
\newcommand{\hN}{\mathcal{N}}

\newcommand{\nn}{\nonumber}
\newcommand{\dt}{{\tau}}
\newcommand{\eps}{\varepsilon}
\newcommand{\daoshu}[2]{\dfrac{\d #1}{\d #2}}

\newtheorem{example}{Example}[section]
\newtheorem{theorem}{Theorem}[section]

\newtheorem{lemma}[theorem]{Lemma}
\newtheorem{proposition}[theorem]{Proposition}
\newtheorem{remark}{Remark}[section]
\newtheorem{scheme}{Scheme}
\journal{Journal of Computational Physics}

\begin{document}

\begin{frontmatter}

\title{Maximum bound principle preserving integrating factor Runge--Kutta methods for semilinear parabolic equations}

\author[usc]{Lili Ju}
\ead{ju@math.sc.edu}

\author[hkpu]{Xiao Li}
\ead{xiao1li@polyu.edu.hk}

\author[hkpu]{Zhonghua Qiao}
\ead{zhonghua.qiao@polyu.edu.hk}

\author[sust]{Jiang Yang}
\ead{yangj7@sustech.edu.cn}

\address[usc]{Department of Mathematics, University of South Carolina, Columbia, SC 29208, USA}
\address[hkpu]{Department of Applied Mathematics, The Hong Kong Polytechnic University, Hung Hom, Kowloon, Hong Kong}
\address[sust]{Department of Mathematics \& SUSTech International Center for Mathematics,
Southern University of Science and Technology, Shenzhen 518055, China}

\begin{abstract}
A large class of semilinear parabolic equations satisfy the maximum bound principle (MBP) in the sense that
the time-dependent solution preserves for any time a uniform pointwise bound  imposed by its initial and boundary conditions.
Investigation on numerical schemes of these equations with preservation of the MBP has attracted increasingly attentions in recent years,
especially for the  temporal discretizations.
In this paper, we study high-order MBP-preserving time integration schemes
by means of the integrating factor Runge--Kutta (IFRK) method.
Beginning with the space-discrete system of semilinear parabolic equations,
we present the  IFRK method in  general form and derive the sufficient conditions for the method to preserve the MBP.
In particular,  we show that
the classic four-stage, fourth-order IFRK scheme is  MBP preserving  for some typical semilinear systems although not strong stability preserving, which can be instantly applied to the Allen--Cahn type of equations. In addition, error estimates for these numerical schemes are proved theoretically and verified numerically, as well as their efficiency  by  simulations of long-time evolutional behavior.
\end{abstract}

\begin{keyword}
maximum bound principle \sep
integrating factor Runge--Kutta method \sep
semilinear parabolic equation \sep
high-order numerical methods \sep
Allen--Cahn equations
\end{keyword}

\end{frontmatter}

\section{Introduction}
\label{sect_intro}


In this paper, we aim to study high-order time discretization methods
to preserve the maximum bound principle (MBP) of a class of semilinear parabolic equations taking the following form
\begin{equation}
\label{model_pde}
u_t = \hL u + \hN[u],
\end{equation}
where $\hL$ and $\hN$ are linear and nonlinear operators, respectively,
and $u=u(t,\bx)$ is the unknown function subject to appropriate initial and boundary conditions.
The MBP implies the existence of special upper and lower solutions in the sense that
the solution $u(t,x)$ preserves for any time a uniform pointwise bound imposed by the initial and boundary data.
Mathematical models for many practical problems have the form \eqref{model_pde}
and their solutions satisfy the MBP.
A typical example is the classic Allen--Cahn equation,
where the linear operator $\hL$ is the standard Laplace operator multiplied by a diffusion coefficient $\eps^2$
and the nonlinear part is given by a cubic polynomial $\hN[u]=u-u^3$.
It is well known \cite{EvSoSo92} that
the solution of the Allen--Cahn equation is pointwisely bounded by $1$ for any time
if the initial and boundary conditions are bounded by $1$,
which implies the MBP.
The Allen--Cahn equation was originally developed in \cite{AlCa79}
to model the motion of anti-phase boundaries in crystalline solids,
and nowadays, has been widely used as a fundamental model for phase-field (or diffuse-interface) methods
to study the interfacial motions and phase transitions in various application fields.

The MBP is an important physical feature
and is essential for numerical simulations to yield physically relevant solutions of many mathematical models.
Besides the solution of the Allen--Cahn equation, the density, concentration, or pressure in fluid flows must be nonnegative,
and the probability distribution should also be in the range $[0,1]$.
Hence, in numerical simulations,
it is highly expected that the numerical solutions preserve the MBP in the discrete sense.
The violation of MBP may cause the ill-posedness of the problem and blow-ups of the numerical algorithms.
For instance, when there is a logarithmic term in the equation, e.g.,
the reaction-diffusion equation with the Flory--Huggins potential function
and the Peng--Robinson equation of state \cite{PeRo76,QiSu14},
a widely used realistic equation of state for hydrocarbon fluid in the petroleum industry.
On the other hand, in the view of numerical analysis,
the MBP suggests a type of strong stability in the supremum-norm sense
and guarantees the spatially uniform pointwise boundedness of the numerical solution.
Such a property facilitates the further numerical analysis of the numerical schemes, e.g., energy stability for phase-field models, since the locally Lipschitz continuous nonlinear term usually will become fully Lipschitz continuous thanks to the uniform pointwise bound. For instance,
the Allen--Cahn equation could be viewed as the $L^2$ gradient flow associating with the energy
\begin{equation}
E(u) = \int \Big( \frac{\eps^2}{2} |\nabla u(\bx)|^2 + \frac{1}{4}(u^2(\bx)-1)^2 \Big) \, \d \bx,
\end{equation}
and the solution satisfies the energy stability
in the sense that the energy is non-increasing in time, namely,
$E(u(t_2))\le E(u(t_1))$ for any $t_2\ge t_1\ge 0$.
There have been a large variety of numerical schemes preserving such energy stability
successfully applied on different types of phase-field models, e.g., \cite{DuJuLiQi18,FeTaYa13,GuWaWi14,JuLiQiZh18,QiZhTa11,ShWaWaWi12,ShXuYa19,ShYa10b,WiWaLo09,XuTa06,Yang16}
and the references therein.
It was found that the MBP plays a key role to show such nonlinear energy stability \cite{DuJuLiQi19,HoTaYa17,ShTaYa16,TaYa16}.


Recently, MBP-preserving numerical schemes
have attracted increasingly attentions for semilinear parabolic equations in the form of \eqref{model_pde}.
In \cite{StVo15}, it was shown for the Allen--Cahn equation in one-dimensional case that
the MBP was preserved by the central difference semi-discrete scheme
and its fully discrete approximations with forward and backward Euler time-stepping methods.
Later, the discrete MBPs, as well as the energy stability,
of the first- and second-order stabilized semi-implicit (SSI) schemes with central difference method
were proved for the Allen--Cahn equation in \cite{HoLe20,ShTaYa16,TaYa16},
and those of the SSI schemes with finite element discretization in space
were obtained for the surface Allen--Cahn equation in \cite{XiFeYu17}.
The discrete MBPs were also obtained for the Allen--Cahn equation
in \cite{DuZh05} by first-order exponential time differencing (ETD) scheme in space-continuous setting,
in \cite{HoXiJi20} by second-order nonlinear implicit-explicit schemes with spatial central difference method,
and in \cite{YaDuZh18} by showing the uniform $L^p$ boundedness and passing the limit as $p$ goes to infinity.
In addition, MBP-preserving numerical schemes have been studied
for the nonlocal Allen--Cahn equation
by using first- and second-order ETD schemes \cite{DuJuLiQi19}
combined with a quadrature-based difference method \cite{DuTaTiYa19},
for the space-fractional Allen--Cahn equation
by using the Crank--Nicolson scheme \cite{HoTaYa17}
with central difference approximation \cite{TiZhDe15},
for the time-fractional Allen--Cahn equation
by using the convex splitting methods \cite{DuYaZh19},
and for the complex-valued Ginzburg--Landau model of superconductivity \cite{DuGuPe92}
by considering the finite volume method \cite{Du98,GaJuXi19}
and finite element method with the mass-lumping technique \cite{Du05} in space
with backward Euler time integration.
In  \cite{tangyang19},
Tang and Yang proposed so-called one-step monotone schemes to check
whether a scheme is an MBP-preserving scheme or not,
in which the regular third-order explicit strong stability-preserving (SSP) method \cite{GoShTa01}
is shown to be MBP-preserving for the Allen--Cahn equation with CFL condition as $\dt=\mathcal{O}(h^2)$.
An abstract framework on the MBP for equations like \eqref{model_pde}
was established in the recent work \cite{DuJuLiQi20review},
where sufficient conditions for linear and nonlinear operators
were given such that the equation satisfies the MBP
and the corresponding first- and second-order ETD schemes preserve the MBP.
Some details on the framework will be given in Section \ref{sect_mbp}.


In this paper, we investigate high-order MBP-preserving time integration schemes.
The integrating factor Runge--Kutta (IFRK) method is used for the time integration,
which has been widely studied for stiff ordinary differential equations (ODEs) recently \cite{AhLi19,JuLiLe14,TaWaNi15}.
The IFRK method could be viewed as an extension of the conventional Runge--Kutta (RK) method,
particularly designed for the case that the problem contains linear part with strong stiffness.
The key idea is to use the exponential integrating factor to eliminate the stiff linear term
and apply the conventional RK method to the resulted system.
To study the stability of the IFRK method,
the concept of strong stability-preserving was proposed in \cite{ShOs88}
to construct efficient time discretization for hyperbolic conservation laws
and then further explored in \cite{GoSh98} for high-order schemes.
The SSP property means that the norm of the numerical solution diminishes.
More precisely, for an ODE system taking the form
\begin{equation}
\label{intro_ode}
\daoshu{u}{t} = Lu + N(u)
\end{equation}
with the matrix $L$ and the mapping $N$ satisfying
\begin{equation}
\label{intro_ode_L}
\|\e^{\omega L}\| \le 1, \quad \forall \, \omega > 0
\end{equation}
and, for some $\omega_0>0$,
\begin{equation}
\label{intro_ode_N}
\|u + \omega N(u)\| \le \|u\|, \quad \forall \, \omega\in(0,\omega_0],
\end{equation}
the IFRK method for \eqref{intro_ode} is called SSP
if its solution satisfies $\|u^{n+1}\|\le\|u^n\|$.
A review of the SSP-RK time discretization method was presented in \cite{GoShTa01}
for ODE systems like \eqref{intro_ode} with $L=0$,
which were often derived from spatial discretization of hyperbolic conservation law equations,
and a recent work \cite{IsGrGo18} generalized these results
to the general case of the SSP-IFRK method for \eqref{intro_ode}.
The general form of the SSP-RK method for \eqref{intro_ode} with $L=0$ is given by
\begin{align}
& u^{(i)} = \sum_{j=0}^{i-1} \left[\alpha_{ij} u^{(j)} + \dt\beta_{ij} N(u^{(j)})\right],
\quad 1\le i\le s,  \label{intro_ode_rk} \\
& \text{with } u^{(0)} = u^n, \ u^{n+1} = u^{(s)}, \nonumber
\end{align}
where $\alpha_{ij}$ are nonnegative.
Here, \eqref{intro_ode_rk} is actually a convex combination
of some forward Euler sub-steps with the step sizes $\dt\frac{\beta_{ij}}{\alpha_{ij}}$
and the nonnegativity of $\beta_{ij}$ is crucial to guarantee \eqref{intro_ode_N}.
It is concluded in \cite{GoSh98,GoShTa01} that,
under the constraint of the nonnegativity of $\beta_{ij}$,
there are no SSP-RK methods with order higher than four
and the number of stage for fourth-order SSP-RK methods cannot be lower than five.


The main contribution of this work includes three aspects.
First, we formulate the general results for the IFRK method preserving the MBP for equations like \eqref{model_pde}
under the abstract framework established in the recent work \cite{DuJuLiQi20review}.
For simplicity, we restrict our discussion in the space-discrete version
to avoid the abstract and tedious definitions of continuous function spaces and domains of operators.
Second, we give the error estimates of the numerical solution of the MBP-preserving IFRK method
by utilizing the uniform $L^\infty$ boundedness guaranteed by the MBP.
Third, we present three-stage, third-order and four-stage, fourth-order MBP-preserving IFRK schemes.
To the best of our knowledge,
this fills the gap of the lack of MBP-preserving numerical schemes with order higher than three.
The requirements on the time step size for preserving the MBP have the same magnitudes as the first-order IF scheme,
and are contributed only from the nonlinear term without the CFL restriction.
Numerical experiments also reflect the high efficiency of the four-stage, fourth-order IFRK scheme.


The rest of this paper is organized as follows.
In Section \ref{sect_mbp},
we briefly restate the sufficient conditions determined in \cite{DuJuLiQi20review}
for the linear and nonlinear operators
such that equation \eqref{model_pde} satisfies the MBP.
We also give the space-discrete equation of \eqref{model_pde}
and the corresponding conditions of the linear and nonlinear parts in the discrete setting.
Then, in Section \ref{sect_ifrk},
we present the IFRK method in the general form
and prove the preservation of the MBP and the error estimates of the method under some certain requirement on the time step size.
In particular, we present a four-stage, fourth-order IFRK scheme which is MBP preserving and give some simple examples of the space-discrete system.
In Section \ref{sect_numexp},
some numerical experiments are carried out for the Allen--Cahn equation with a logarithmic nonlinear term,
including the tests of convergence rate, MBP, and  efficiency for long-time simulations.
Finally, some concluding remarks are given in Section \ref{sect_con}.

\section{Maximum bound principle for semilinear parabolic equations}
\label{sect_mbp}

An abstract framework on the maximum bound principle
for a class of semilinear parabolic equations \eqref{model_pde}
was established in \cite{DuJuLiQi20review},
where sufficient conditions for the linear and nonlinear operators
were given such that equation \eqref{model_pde} satisfies the MBP.
For the completeness of the current paper,
we present some main results in \cite{DuJuLiQi20review}.
Denote by $\Omega$ the spatial domain as usual.

A crucial condition on the linear operator $\hL$ is the dissipativity in the sense that
if a function $w$ reaches its maximum on $\overline{\Omega}$ at a point $x_0\in\Omega$,
then it must hold $\hL w(\bx_0)\le0$.
By defining the function space $X$ appropriately,
this condition implies that $\hL$ is the generator of a contraction semigroup on $X$.
Such $\hL$ could be the standard Laplace operator, nonlocal diffusion operator \cite{DuGuLeZh12},
fractional Laplace operator \cite{NePaVa12} and so on.
The nonlinear operator $\hN$ is assumed to act as a composite function,
i.e., $\hN[w](\bx)=f(w(\bx))$ for any function $w$ and $\bx\in\overline{\Omega}$,
where $f$ is a one-variable continuously differentiable function satisfying
\begin{equation}
\label{assump_f}
f(\rho)\le 0\le f(-\rho), \quad \text{for some constant $\rho>0$.}
\end{equation}
The essential is the change of the sign of $f$ on both sides of zero.
Under these assumptions,
equation \eqref{model_pde} satisfies the MBP, that is,
if the absolute values of initial and boundary conditions are bounded by $\rho$,
then the absolute value of the entire solution is also bounded by $\rho$ pointwisely for all time.

Applying some type of spatial discretization to \eqref{model_pde},
one can obtain the space-discrete problem given by the ordinary differential equation (ODE) system
\begin{equation}
\label{model_eq}
\daoshu{u}{t} = Lu + f(u).
\end{equation}
Here, $u(t)=(u_1(t),u_2(t),\dots,u_m(t))^T\in\R^m$ denotes the space-discrete solution,
$L$ is an $m$-by-$m$ symmetric matrix derived from the spatial discretization of the linear operator $\hL$,
and the vector $f(u)$ with the $j$-th component $f(u_j)$ corresponds to the nonlinear term $\hN[u]$.
We denote by $\|\cdot\|_\infty$ the vector or matrix $\infty$-norm as usual.
The framework developed in \cite{DuJuLiQi20review} also consists of the space-discrete problem \eqref{model_eq}.
Therefore, we require that the matrix $L$ is the generator of a contraction semigroup on $\R^m$,
or equivalently,
\begin{equation}
\label{cond_L}
\|\e^{\omega L}\|_\infty \le 1, \quad \forall \, \omega > 0,
\end{equation}
which is identical to \eqref{intro_ode_L}.
For the nonlinear function $f$,
due to the assumption \eqref{assump_f}, i.e., the change of the sign of $f$ on both sides of zero,
the following condition holds:
\begin{equation}
\label{cond_f}
\text{$\exists\,\omega_0^+>0$ such that
$|\xi+\omega f(\xi)| \le \rho$, $\forall\,\xi \in [-\rho,\rho]$, $\forall\,\omega\in(0,\omega_0^+]$,}
\end{equation}
which is weaker than \eqref{intro_ode_N}.
Sometimes, we also need to further assume:
\begin{equation}
\label{cond_f2}
\text{$\exists\,\omega_0^->0$ such that
$|\xi-\omega f(\xi)| \le \rho$, $\forall\,\xi \in [-\rho,\rho]$, $\forall\,\omega\in(0,\omega_0^-]$.}
\end{equation}

\begin{remark}
The above two conditions \eqref{cond_f}-\eqref{cond_f2} are very crucial to remove the nonnegativity requirement of the coefficients $\beta_{ij}$s
as we mentioned in the standard SSP-RK schemes \eqref{intro_ode_rk}. This allows us to show that the three-stage, third-order and four-stage, fourth-order IFRKs are actually MBP-preserving for the model equation \eqref{model_eq} satisfying \eqref{assump_f},  which will be demonstrated in the following section.
\end{remark}

\section{Integrating factor Runge--Kutta method}
\label{sect_ifrk}

In this section,
we will present a family of IFRK schemes for time-stepping of the space-discrete system \eqref{model_eq}.
The method is based on the Runge--Kutta time discretizations
combined with the exponential integrating factor.
We will show the MBP-preserving property under the conditions \eqref{cond_L}--\eqref{cond_f2}
and the error estimates of the numerical solutions.

\subsection{MBP-preserving IFRK method in general form}

We have claimed that $L$ in \eqref{model_eq} is an $m$-by-$m$ matrix
corresponding to the spatial discretization of the linear operator $\hL$.
Premultiplying the system \eqref{model_eq} by $\e^{-tL}$, we have
\[
\daoshu{(\e^{-tL}u)}{t} = \e^{-tL} f(u).
\]
Defining a transformation of variables by $w = \e^{-tL} u$,
we obtain a new ODE system
\begin{equation}
\label{model_eq_w}
\daoshu{w}{t} = \e^{-tL} f(\e^{tL}w) =: G(t,w).
\end{equation}
The general explicit $s$-stage Runge--Kutta method for \eqref{model_eq_w} is given by \cite{ShOs88}
\begin{subequations}
\label{wODE_RK_general}
\begin{align}
w^{(0)} & = w^n,\\
w^{(i)} & = w^{(0)}+\dt\sum_{j=0}^{i-1}d_{ij}G(t_n+c_j\dt,w^{(j)}),\quad 1\le i\le s, \label{wODE_RK_general2} \\
w^{n+1} & = w^{(s)},
\end{align}
\end{subequations}
where $c_0=0$, $c_i=\sum\limits_{j=0}^{i-1}d_{ij}$ for $1\le i\le s$, 
and $c_s=1$ for consistency.
For $\alpha_{ij}\ge0$ to be determined such that $\sum\limits_{j=0}^{i-1}\alpha_{ij}=1$,
we rewrite \eqref{wODE_RK_general2} as
\begin{equation}
\label{wODE_RK_euler}
w^{(i)} = \sum_{j=0}^{i-1}[\alpha_{ij}w^{(j)}+\dt\beta_{ij}G(t_n^{(j)},w^{(j)})],\quad 1\le i\le s,
\end{equation}
where $\beta_{ij}=d_{ij}-\sum\limits_{k=j+1}^{i-1}\alpha_{ik}d_{kj}$ and $t_n^{(j)}=t_n+c_j\dt$.
If we require that $\beta_{ij}=0$  when its corresponding $\alpha_{ij}=0$, thus
\eqref{wODE_RK_euler} is a convex combination of a group  of forward Euler substeps
$$ w^{(j)}+\dt\frac{\beta_{ij}}{\alpha_{ij}}G(t_n^{(j)},w^{(j)}).$$
Define $u^n=\e^{t_nL}w^n$ and $u^{(i)}=\e^{t_n^{(i)}L}w^{(i)}$, $0\le i\le s$,
then \eqref{wODE_RK_euler} becomes
\begin{equation*}
u^{(i)} = \sum_{j=0}^{i-1} \e^{(c_i-c_j)\dt L} [\alpha_{ij} u^{(j)} + \dt\beta_{ij} f(u^{(j)})],
\quad 1\le i\le s,
\end{equation*}
which can be viewed as
a convex combination of the exponential forward Euler substeps
\[
\e^{(c_i-c_j)\dt L} \Big[u^{(j)}+\dt\frac{\beta_{ij}}{\alpha_{ij}}f(u^{(j)})\Big].
\]
Now, we obtain the following $s$-stage IFRK method for \eqref{model_eq}:
\begin{subequations}
\label{model_ODE_IFRK}
\begin{align}
u^{(0)} & = u^n,\\
u^{(i)} & = \sum_{j=0}^{i-1} \e^{(c_i-c_j)\dt L} [\alpha_{ij} u^{(j)} + \dt\beta_{ij} f(u^{(j)})],
\quad 1\le i\le s,\\
u^{n+1} & = u^{(s)}.
\end{align}
\end{subequations}
The main result on the MBP-preserving property of the method \eqref{model_ODE_IFRK} is as follows.

\begin{theorem}
\label{thm_MBP}
Given a linear operator $L$ satisfying \eqref{cond_L},
a function $f$ satisfying \eqref{cond_f} and \eqref{cond_f2},
and the abscissas $\{c_j\}$ satisfying
\begin{equation}
\label{cond_cc}
c_0 \le c_1 \le \cdots \le c_s,
\end{equation}
if $\|u^n\|_\infty\le\rho$,
then $u^{n+1}$ obtained from \eqref{model_ODE_IFRK} satisfies $\|u^{n+1}\|_\infty \le \rho$,
provided that the time step size satisfies
\begin{equation}
\label{cond timestep}
\dt \le \mathcal{C}\omega_0^+, \quad \text{with } \mathcal{C}=\min_{i,j}\frac{\alpha_{ij}}{\beta_{ij}}
\end{equation}
when $\beta_{ij}$ are all nonnegative,
or satisfies both \eqref{cond timestep} and
\begin{equation}
\label{cond timestep2}
\dt \le \mathcal{C}\omega_0^-, \quad \text{with } \mathcal{C}=\min_{i,j}\frac{\alpha_{ij}}{|\beta_{ij}|}
\end{equation}
whenever there is a negative $\beta_{ij}$.
\end{theorem}

\begin{proof}
For each stage of \eqref{model_ODE_IFRK}, suppose $\|u^{(j)}\|_\infty\le\rho$ for all $j\le i-1$.
Then, we have
\begin{align*}
\|u^{(i)}\|_\infty
& \le \sum_{j=0}^{i-1} \|\e^{(c_i-c_j)\dt L} [\alpha_{ij} u^{(j)} + \dt\beta_{ij} f(u^{(j)})]\|_\infty\\
& \le \sum_{j=0}^{i-1}\alpha_{ij}\|\e^{(c_i-c_j)\dt L}\|_\infty
\Big\|u^{(j)} + \dt\frac{\beta_{ij}}{\alpha_{ij}} f(u^{(j)})\Big\|_\infty\\
& \le \sum_{j=0}^{i-1}\alpha_{ij} \cdot \rho\\
& = \rho,
\end{align*}
since $c_i-c_j\ge0$,
and $\dt \max\limits_{i,j}\frac{\beta_{ij}}{\alpha_{ij}}\le\omega_0^+$
or $\dt \max\limits_{i,j}\frac{|\beta_{ij}|}{\alpha_{ij}}\le\min\{\omega_0^+,\omega_0^-\}$.
By induction, we obtain $\|u^{(i)}\|_\infty\le\rho$ for each $i$, and thus $\|u^{n+1}\|\le\rho$.
\end{proof}

\begin{remark}
Condition \eqref{cond_cc} implies the property of \emph{non-decreasing abscissas},
which is crucial for the preservation of the MBP for the IFRK method \eqref{model_ODE_IFRK}.
\end{remark}

\subsection{Error estimate}

To carry out convergence analysis for the IFRK method \eqref{model_ODE_IFRK},
we transform the variable $w$ in \eqref{wODE_RK_general} back to $u$ to get
\begin{subequations}
\label{IFRK_Butcher}
\begin{align}
u^{(0)} & = u^n,\\
u^{(i)} & = \e^{c_i\dt L} u^n + \dt \sum_{j=0}^{i-1} d_{ij} \e^{(c_i-c_j)\dt L} f(u^{(j)}),
\quad 1\le i\le s-1, \label{IFRK_Butcher2} \\
u^{n+1} & = \e^{\dt L} u^n + \dt \sum_{i=0}^{s-1} d_{si} \e^{(1-c_i)\dt L} f(u^{(i)}). \label{IFRK_Butcher3}
\end{align}
\end{subequations}
For simplicity, we do not introduce the general order conditions for arbitrary-order RK methods (see, e.g., \cite{HairerNoWa93}).
Instead, we suppose directly that the RK method \eqref{wODE_RK_general} is $p$-th order, where $1\le p\le s$.
Then, we have the following error estimate for \eqref{IFRK_Butcher}.

\begin{theorem}
Given $T>0$,
suppose that the exact solution, denoted by $u_e(t)$, of \eqref{model_eq} is sufficiently smooth on $[0,T]$
and $f$ is $p$-times continuously differentiable on $[-\rho,\rho]$.
Under the conditions of Theorem \ref{thm_MBP},
if the time step size $\dt$ satisfies \eqref{cond timestep} and \eqref{cond timestep2}, then
the numerical solution $u^n$ generated by the IFRK method \eqref{IFRK_Butcher}
with $u^0=u_e(0)$ and $\|u_e(0)\|_\infty\le\rho$
satisfies the error estimate:
\[
\|u_e(t_n)-u^n\|_\infty \le C(\e^{F_1st_n}-1)\dt^p, \quad t_n\le T,
\]
where $F_1=\max\limits_{|\xi|\le\rho}|f'(\xi)|$ and
the constant $C>0$ 
is independent of $\dt$.
\end{theorem}

\begin{proof}
Following \cite{DuJuLu19,Ying00},
let us introduce the reference functions $v^{(i)}$ satisfying
\begin{subequations}
\begin{align}
v^{(0)} & = u_e(t_n), \label{IFRK_Butcher_v1} \\
v^{(i)} & = \e^{c_i\dt L} u_e(t_n) + \dt \sum_{j=0}^{i-1} d_{ij} \e^{(c_i-c_j)\dt L} f(v^{(j)}),
\quad 1\le i\le s-1, \label{IFRK_Butcher_v2} \\
u_e(t_{n+1}) & = \e^{\dt L} u_e(t_n) + \dt \sum_{i=0}^{s-1} d_{si} \e^{(1-c_i)\dt L} f(v^{(i)}) + \dt R^n, \label{IFRK_Butcher_v3}
\end{align}
\end{subequations}
where the truncation error $R^n$ satisfies
\begin{equation}
\label{thm_error_pf0}
\max_{0\le n\le [T/\dt]} \|R^n\|_\infty \le \widetilde{C} \dt^p,
\end{equation}
where the constant $\widetilde{C}>0$ depends on the $C^{p+1}[0,T]$-norm of $u_e$,
the $C^p[-\rho,\rho]$-norm of $f$, $\|L\|_\infty$, $T$, $p$, and $s$,
but is independent of $\dt$.
Note that $\|u_e(t)\|_\infty\le\rho$ for any $t\in[0,T]$ due to the MBP of \eqref{model_eq}.
According to \eqref{IFRK_Butcher_v1} and \eqref{IFRK_Butcher_v2},
we know from the proof of Theorem \ref{thm_MBP} that $\|v^{(i)}\|_\infty\le\rho$ for each $i=1,2,\dots,s-1$.
Let $e^n=u_e(t_n)-u^n$ and $e^{(i)}=v^{(i)}-u^{(i)}$ for $i=0,1,\dots,s-1$.

For each $i=1,2,\dots,s-1$, the difference between \eqref{IFRK_Butcher_v2} and \eqref{IFRK_Butcher2} gives
\[
e^{(i)} = \e^{c_i\dt L} e^n + \dt \sum_{j=0}^{i-1} d_{ij} \e^{(c_i-c_j)\dt L} [f(v^{(j)}) - f(u^{(j)})].
\]
Using \eqref{cond_L} and noting $d_{ij}\le c_i\le 1$, we obtain
\[
\|e^{(i)}\|_\infty
\le \|e^n\|_\infty + \dt \sum_{j=0}^{i-1} \|f(v^{(j)}) - f(u^{(j)})\|_\infty
\le \|e^n\|_\infty + F_1 \dt \sum_{j=0}^{i-1} \|e^{(j)}\|_\infty.
\]
By induction, assuming $\|e^{(j)}\|_\infty\le(1+F_1\dt)^j\|e^n\|_\infty$ for $j=0,1,\dots,i-1$,
we obtain
\begin{equation}
\label{thm_error_pf1}
\|e^{(i)}\|_\infty
\le \|e^n\|_\infty + F_1 \dt \sum_{j=0}^{i-1} (1+F_1\dt)^j \|e^n\|_\infty
= (1+F_1\dt)^i \|e^n\|_\infty.
\end{equation}
Therefore, the inequality \eqref{thm_error_pf1} holds for any $i=0,1,2,\dots,s-1$.

Similarly, the difference between \eqref{IFRK_Butcher_v3} and \eqref{IFRK_Butcher3} leads to
\[
e^{n+1} = \e^{\dt L} e^n + \dt \sum_{i=0}^{s-1} d_{si} \e^{(1-c_i)\dt L} [f(v^{(i)}) - f(u^{(i)})] + \dt R^n,
\]
and then, using \eqref{thm_error_pf1} and \eqref{thm_error_pf0}, it yields
\begin{align*}
\|e^{n+1}\|_\infty
& \le \|e^n\|_\infty + F_1 \dt \sum_{i=0}^{s-1} \|e^{(i)}\|_\infty + \dt \|R^n\|_\infty \\
& \le \|e^n\|_\infty + F_1 \dt \sum_{i=0}^{s-1} (1+F_1\dt)^i \|e^n\|_\infty + \widetilde{C} \dt^{p+1} \\
& = (1+F_1\dt)^s \|e^n\|_\infty + \widetilde{C} \dt^{p+1}.
\end{align*}
By recursion, we obtain
\[
\|e^n\|_\infty \le (1+F_1\dt)^{ns} \|e^0\|_\infty + \widetilde{C} \dt^{p+1} \sum_{k=0}^{n-1} (1+F_1\dt)^{ks}.
\]
Noting that $e^0=0$ and denoting $C=\widetilde{C}(F_1s)^{-1}$, we have
\[
\|e^n\|_\infty \le C \dt^{p} [(1+F_1\dt)^{ns}-1] \le C (\e^{F_1st_n}-1) \dt^{p},
\]
which completes the proof.
\end{proof}

\subsection{Various MBP-preserving IFRK schemes}
\label{sect_IFRKschemes}

We have shown the MBP-preserving property for the IFRK method in the general form.
Now, we present some concrete and practical MBP-preserving IFRK schemes
under the general results established above.

\begin{scheme}[IF1]
The first-order integrating factor (IF1) scheme
for solving \eqref{model_eq} reads \cite{IsGrGo18}
\begin{equation}
\label{IFFE}
u^{n+1} = \e^{\dt L} [u^n + \dt f(u^n)].
\end{equation}
Here, $\beta_{ij}>0$ and $\mathcal{C}=1$.
Thus, the condition \eqref{cond timestep} becomes
\begin{equation}
\label{cond_timestep_IF1}
\dt \le \omega_0^+.
\end{equation}
\end{scheme}

\begin{scheme}[IFRK2]
A second-order integrating factor Runge--Kutta (IFRK2) scheme
for solving \eqref{model_eq} reads \cite{IsGrGo18}
\begin{subequations}
\label{IFRK2}
\begin{align}
u^{(1)} & = \e^{\dt L} [u^n + \dt f(u^n)] \nn \\
& = \e^{\dt L} u^n + \dt \e^{\dt L} f(u^n), \\
u^{n+1} & = \frac{1}{2} \e^{\dt L} u^n + \frac{1}{2} [u^{(1)} + \dt f(u^{(1)})] \nn \\
& = \e^{\dt L} u^n + \dt\Big(\frac{1}{2}\e^{\dt L} f(u^n) + \frac{1}{2} f(u^{(1)})\Big).
\end{align}
\end{subequations}
Here, $\beta_{ij}\ge0$ and $\mathcal{C}=1$.
Thus, the condition \eqref{cond timestep} is the same as \eqref{cond_timestep_IF1}.
\end{scheme}

\begin{scheme}[IFRK3]
A third-order integrating factor Runge--Kutta (IFRK3) scheme
for solving \eqref{model_eq} reads \cite{IsGrGo18}
\begin{subequations}
\label{IFRK3}
\begin{align}
u^{(1)} & = \frac{1}{2} \e^{\frac{2\dt}{3} L}u^n
+ \frac{1}{2} \e^{\frac{2\dt}{3} L} \Big[u^n + \frac{4\dt}{3}f(u^n)\Big] \nn \\
& = \e^{\frac{2\dt}{3}L} u^n + \frac{2\dt}{3} \e^{\frac{2\dt}{3}L} f(u^n),\\
u^{(2)} & = \frac{2}{3} \e^{\frac{2\dt}{3}L}u^n
+ \frac{1}{3} \Big[u^{(1)} + \frac{4\dt}{3}f(u^{(1)})\Big] \nn \\
& = \e^{\frac{2\dt}{3}L}u^n
+ \frac{2\dt}{3} \Big(\frac{1}{3}\e^{\frac{2\dt}{3}L}f(u^n)+\frac{2}{3}f(u^{(1)})\Big), \\
u^{n+1} & = \frac{59}{128}\e^{\dt L}u^n + \frac{15}{128}\e^{\dt L} \Big[u^n + \frac{4\dt}{3}f(u^n)\Big]
+ \frac{27}{64}\e^{\frac{\dt}{3}L} \Big[u^{(2)} + \frac{4\dt}{3}f(u^{(2)})\Big] \nn \\
& = \e^{\dt L}u^n
+ \dt\Big(\frac{4}{16}\e^{\dt L}f(u^n) + \frac{3}{16}\e^{\frac{\dt}{3}L}f(u^{(1)}) + \frac{9}{16}\e^{\frac{\dt}{3}L}f(u^{(2)})\Big).
\end{align}
\end{subequations}
Here, $\beta_{ij}\ge0$ and $\mathcal{C}=\dfrac{3}{4}$.
Thus, the condition \eqref{cond timestep} leads to
\begin{equation}
\label{cond_timestep_IFRK3}
\dt \le \frac{3\omega_0^+}{4}.
\end{equation}
\end{scheme}

\begin{remark}
\label{rmk_ShuOsher}
Applying the third-order Shu--Osher method \cite{ShOs88} to \eqref{model_eq_w} gives
\begin{subequations}
\label{IFRK3n}
\begin{align}
u^{(1)} & = \e^{\dt L} [u^n + \dt f(u^n)] \nn \\
& = \e^{\dt L} u^n + \dt \e^{\dt L} f(u^n), \\
u^{(2)} & = \frac{3}{4} \e^{\frac{\dt}{2}L}u^n + \frac{1}{4} \e^{-\frac{\dt}{2}L} [u^{(1)} + \dt f(u^{(1)})] \nn \\
& = \e^{\frac{\dt}{2}L}u^n
+ \frac{\dt}{2} \Big(\frac{1}{2}\e^{\frac{\dt}{2}L} f(u^n) + \frac{1}{2}\e^{-\frac{\dt}{2}L} f(u^{(1)})\Big), \\
u^{n+1} & = \frac{1}{3}\e^{\dt L}u^n + \frac{2}{3}\e^{\frac{\dt}{2}L}[u^{(2)}+\dt f(u^{(2)})] \nn \\
& = \e^{\dt L}u^n
+ \dt \Big(\frac{1}{6}\e^{\dt L}f(u^n) + \frac{2}{3}\e^{\frac{\dt}{2}L}f(u^{(2)}) + \frac{1}{6}f(u^{(1)})\Big),
\end{align}
\end{subequations}
which gives $\beta_{ij}\ge0$ and $\mathcal{C}=1$.
However, there is a matrix exponential with a negative coefficient,
so this scheme may not preserve the MBP.
We will show in Section \ref{sect_numexp} that
the scheme \eqref{IFRK3n} does not preserve the MBP
even though a small time step size is used.
This suggests the necessity of the property of non-decreasing abscissas.
\end{remark}

Apart from the three-stage IFRK3 scheme \eqref{IFRK3},
one can give more MBP-preserving third-order IFRK schemes
by combining the exponential integrating factor approach with the RK method with non-decreasing abscissas,
e.g., eSSPRK$^+$(4,3) in \cite{IsGrGo18}.

\begin{scheme}[IFRK4]
Applying the classic fourth-order Runge--Kutta method to \eqref{model_eq_w} gives
the fourth-order integrating factor Runge--Kutta (IFRK4) scheme:
\begin{subequations}
\label{IFRK4}
\begin{align}
u^{(1)} & = \e^{\frac{\dt}{2}L} \Big[u^n + \frac{\dt}{2} f(u^n)\Big] \nn \\
& = \e^{\frac{\dt}{2}L} u^n + \frac{\dt}{2} \e^{\frac{\dt}{2}L} f(u^n), \\
u^{(2)} & = \frac{1}{2} \e^{\frac{\dt}{2}L} \Big[u^n - \frac{\dt}{2} f(u^n)\Big]
+ \frac{1}{2} [u^{(1)} + \dt f(u^{(1)})] \nn \\
& = \e^{\frac{\dt}{2}L} u^n + \frac{\dt}{2} f(u^{(1)}), \\
u^{(3)} & = \frac{1}{9} \e^{\dt L} [u^n - \dt f(u^n)]
+ \frac{2}{9} \e^{\frac{\dt}{2}L} \Big[u^{(1)} - \frac{3\dt}{2} f(u^{(1)})\Big]
+ \frac{2}{3} \e^{\frac{\dt}{2}L} \Big[u^{(2)} + \frac{3\dt}{2} f(u^{(2)})\Big] \nn \\
& = \e^{\dt L} u^n + \dt \e^{\frac{\dt}{2}L} f(u^{(2)}), \\
u^{n+1} & = \frac{1}{3} \e^{\frac{\dt}{2}L} \Big[u^{(1)} + \frac{\dt}{2} f(u^{(1)})\Big]
+ \frac{1}{3} \e^{\frac{\dt}{2}L} u^{(2)}
+ \frac{1}{3} \Big[u^{(3)} + \frac{\dt}{2} f(u^{(3)})\Big] \nn \\
& = \e^{\dt L} u^n + \dt \Big(\frac{1}{6}\e^{\dt L}f(u^n) + \frac{1}{3}\e^{\frac{\dt}{2}L}f(u^{(1)})
+ \frac{1}{3}\e^{\frac{\dt}{2}L}f(u^{(2)}) + \frac{1}{6}f(u^{(3)})\Big).
\end{align}
\end{subequations}
Here, $\mathcal{C}=\dfrac{2}{3}$ and $\beta_{ij}$ are not all nonnegative.
Thus, the conditions \eqref{cond timestep} and \eqref{cond timestep2} yield
\begin{equation}
\label{cond_timestep_IFRK4}
\dt \le \frac{2\omega_0^*}{3},
\end{equation}
where $\omega_0^*=\min\{\omega_0^+,\omega_0^-\}>0$.
\end{scheme}

As illustrated in \cite{GoSh98},
the constraint of the nonnegativity of $\beta_{ij}$
leads to the nonexistence of four-stage, fourth-order SSP-IFRK schemes.
In our work, however, the
conditions \eqref{cond_f} and \eqref{cond_f2} relax the requirement of the nonnegative $\beta_{ij}$.
Therefore, the IFRK4 scheme \eqref{IFRK4} is adequate to preserve the MBP
without any needs of extra computations for the modification of $f$ as done in \cite{ShOs88}.
More MBP-preserving fourth-order IFRK schemes with larger numbers of stage
could be obtained by using, for example, eSSPRK$^+$(5,4) and eSSPRK$^+$(10,4) presented in \cite{IsGrGo18}.

Since $\omega_0^+$ and $\omega_0^-$ are completely determined by $f$ via \eqref{cond_f} and \eqref{cond_f2},
from the inequalities \eqref{cond_timestep_IF1}, 
\eqref{cond_timestep_IFRK3}, and \eqref{cond_timestep_IFRK4},
we find that the constraints on the time step sizes depend only on $f$ but not on the size of $L$.
This means that the choice of the time step sizes is independent of the spatial mesh size.
Moreover, we point out that
these constraints are all sufficient but not necessary conditions for the MBP-preserving property.

\subsection{Examples of the matrix $L$ and the function $f$}

A large number of examples of linear and nonlinear operators in \eqref{model_pde}
have been shown to satisfy the assumptions made in \cite{DuJuLiQi20review},
including the space-continuous and space-discrete cases,
and thus, we can check the matrix $L$ in \eqref{model_eq} in the same way.
Here, we present a more direct criterion for $L$ from the point of view of matrices.

\begin{lemma}[\cite{Dahlquist58,Soderlind06}]
\label{lem_expnorm}
For any matrix $A=(a_{ij})\in\R^{m\times m}$ and any constant $s\ge 0$, we have
\[
\|\e^{sA}\|_\infty \le \e^{s\mu_\infty(A)},
\]
where $\mu_\infty(A)$ is the logarithmic norm of $A$ with respect to the $\infty$-norm, i.e.,
\[
\mu_\infty(A) = \max_{1\le i\le m} \bigg(a_{ii} + \sum_{\substack{j=1\\ j\not=i}}^m|a_{ij}|\bigg).
\]
\end{lemma}

\begin{remark}
If $A$ is diagonally dominant with all diagonal entries negative,
then we have $\|\e^{sA}\|_\infty \le 1$ for any $s\ge0$ since $\mu_\infty(A)\le0$.
See the following example.
\end{remark}

\begin{example}
\rm If $L$ is given by the second-order central difference discretization of $\Delta$, i.e.,
\[
L=\frac{1}{h^2}
\begin{pmatrix}
-2 & 1 & ~ & ~ & c\\
1 & -2 & 1 \\
~ & \ddots & \ddots & \ddots \\
~ & ~ & 1 & -2 & 1 \\
c & ~ & ~ & 1 & -2
\end{pmatrix}
\quad\text{with $c=0$ or $c=1$},
\]
then $\mu_\infty(L)=0$ and $\mu_\infty(-L)=4/h^2$.
According to Lemma \ref{lem_expnorm}, we have
\[
\|\e^{\dt L}\|_\infty \le 1, \qquad \|\e^{-\dt L}\|_\infty \le \e^{\frac{4\dt}{h^2}}.
\]
Denoting by $I$ the identity matrix with the same size as $L$ and letting
\begin{equation}
\label{eg_L2}
L^{(2)} = I \otimes L + L \otimes I,
\end{equation}
we have $\mu_\infty(L^{(2)})=0$ and $\mu_\infty(-L^{(2)})=8/h^2$, and thus,
\begin{equation}
\label{eg_Lexp}
\|\e^{\dt L^{(2)}}\|_\infty \le 1, \qquad \|\e^{-\dt L^{(2)}}\|_\infty \le \e^{\frac{8\dt}{h^2}}.
\end{equation}
The three-dimensional case $L^{(3)} = I \otimes I \otimes L + I \otimes L \otimes I + L \otimes I \otimes I$ is quite similar.
\end{example}

To guarantee \eqref{cond_f} and  \eqref{cond_f2} for the nonlinear function $f$,
we actually have the following result.
The proof is straightforward, so we omit it.

\begin{proposition}
If there exists $\rho>0$ such that $f(\pm\rho)=0$
and $f$ is continuously differentiable and nonconstant on $[-\rho,\rho]$,
then \eqref{cond_f} and \eqref{cond_f2} hold respectively for
\[
\omega_0^+ = - \frac{1}{\min\limits_{|\xi|\le\rho}f'(\xi)} \quad \text{and} \quad
\omega_0^- = \frac{1}{\max\limits_{|\xi|\le\rho}f'(\xi)}.
\]
\end{proposition}

\begin{example}
\rm The Allen--Cahn equation has the nonlinear term $f(u)=u-u^3$,
which satisfies \eqref{cond_f} and \eqref{cond_f2}  with $\rho=1,\omega_0^+=\frac{1}{2}$, and $\omega_0^-=1$.
\end{example}

\begin{example}
\rm Given the Flory--Huggins potential function
\begin{equation}
\label{F_FloryHuggins}
F(u) = \frac{\theta}{2} [(1+u)\ln(1+u) + (1-u)\ln(1-u)] - \frac{\theta_c}{2}u^2,
\end{equation}
where $\theta$ and $\theta_c$ are two positive constants satisfying $\theta<\theta_c$.
We set $f(u)=-F'(u)$, namely,
\begin{equation}
\label{f_log}
f(u) = \frac{\theta}{2}\ln\frac{1-u}{1+u} + \theta_cu.
\end{equation}
Denote by $\gamma$ the positive root of $f(\gamma)=0$.
Noting that
\[
\max_{|\xi|\le\gamma}f'(\xi) = \theta_c - \theta > 0, \qquad
\min_{|\xi|\le\gamma}f'(\xi) = \theta_c - \frac{\theta}{1-\gamma^2} < 0,
\]
we know $f$ satisfies \eqref{cond_f} and \eqref{cond_f2} with $\rho=\gamma$,
$\omega_0^+=\frac{1-\gamma^2}{\theta-\theta_c(1-\gamma^2)}$, and $\omega_0^-=\frac{1}{\theta_c-\theta}$.
\end{example}

\section{Numerical experiments}
\label{sect_numexp}

Let us consider the two-dimensional  reaction-diffusion equation
\begin{equation}
\label{eg_AllenCahn}
u_t = \eps^2\Delta u + f(u), \quad (x,y)\in\Omega=(0,1)^2, \ t\in(0,T],
\end{equation}
subject to the periodic boundary condition,
where $f(u)$ takes the form \eqref{f_log} with $\theta=0.8$ and $\theta_c=1.6$.
The positive root $\gamma$ of $f(\gamma)=0$ is approximately $\gamma\approx0.9575$.
The energy functional corresponding to \eqref{eg_AllenCahn} is given by
\begin{equation}
\label{eg_energy}
E(u) = \int_{(0,1)^2} \Big( \frac{\eps^2}{2}|\nabla u(\bx)|^2 + F(u(\bx)) \Big) \, \d \bx,
\end{equation}
where $F(u)$ is the Flory--Huggins potential \eqref{F_FloryHuggins}.
In all experiments,
we always adopt the five-point central difference matrix \eqref{eg_L2} to approximate the Laplace operator
on the spatial uniform mesh with the size $h$ given later.
Since the approximating matrix is circulant,
the product of the matrix exponential and a vector is calculated via the fast Fourier transform.

\subsection{Convergence tests}

First, we test the convergence rates of the IFRK schemes \eqref{IFFE}, \eqref{IFRK2}, \eqref{IFRK3}, and \eqref{IFRK4}.
The initial data is set to be
\[
u_0(x,y) = 0.1 (\sin 3\pi x \sin 2\pi y + \sin 5\pi x \sin 5\pi y).
\]
We use the spatial mesh size $h=1/2048$.
For the cases $\eps=0.1$ and $\eps=0.01$,
we calculate the numerical solutions at $T=2$ with the time step sizes $\dt=2^{-k}$, $k=1,2,\dots,12$
and regard the solution obtained by IFRK4 with $\dt=0.1\times 2^{-12}$ as the benchmark
to compare the supremum-norm errors.
\figurename~\ref{fig_convergence} shows the results of the convergence tests
and the expected convergence rates are obvious.

\begin{figure}[h]
\centering
\subfigure[IF1 scheme]{\includegraphics[width=0.5\textwidth]{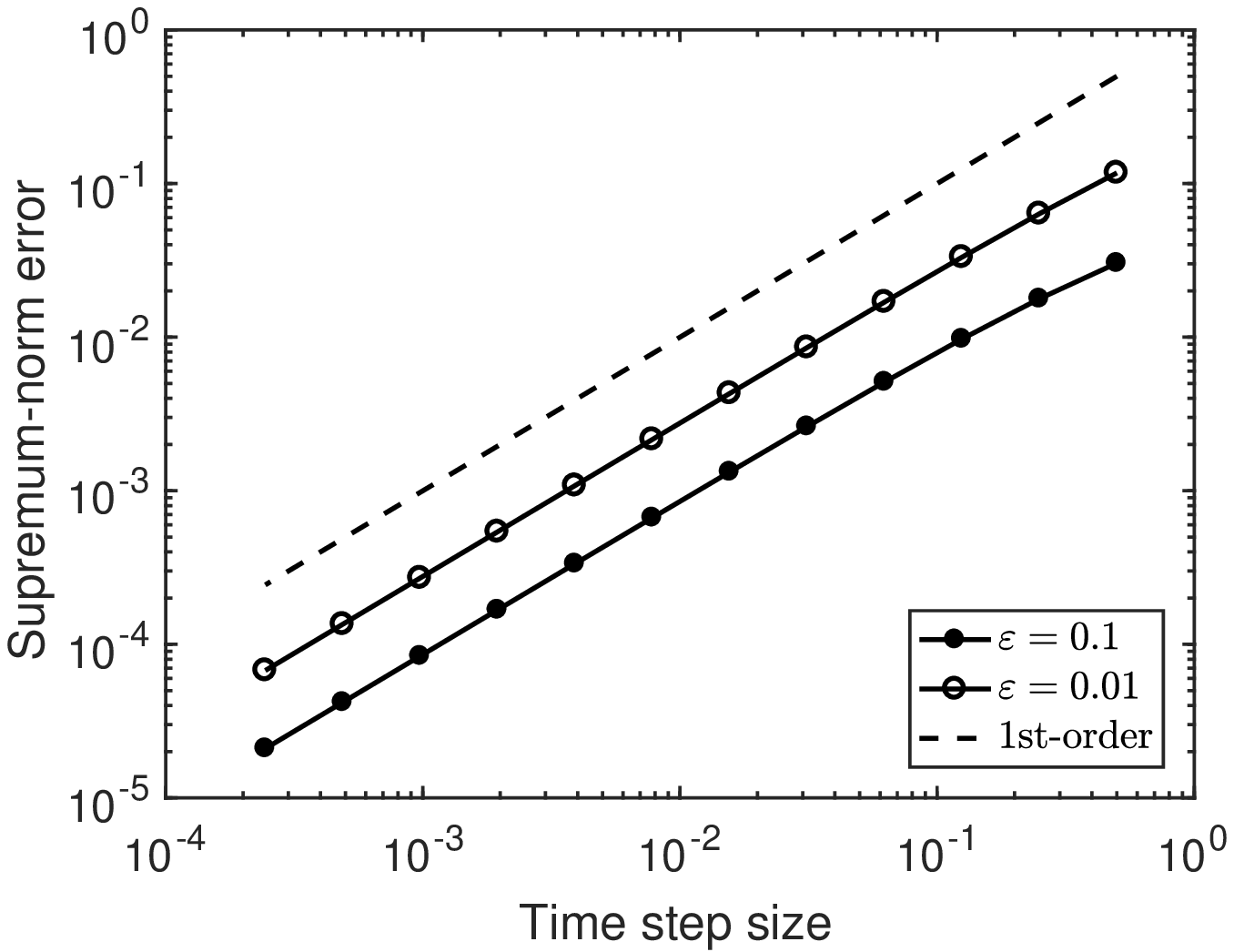}}
\subfigure[IFRK2 scheme]{\includegraphics[width=0.5\textwidth]{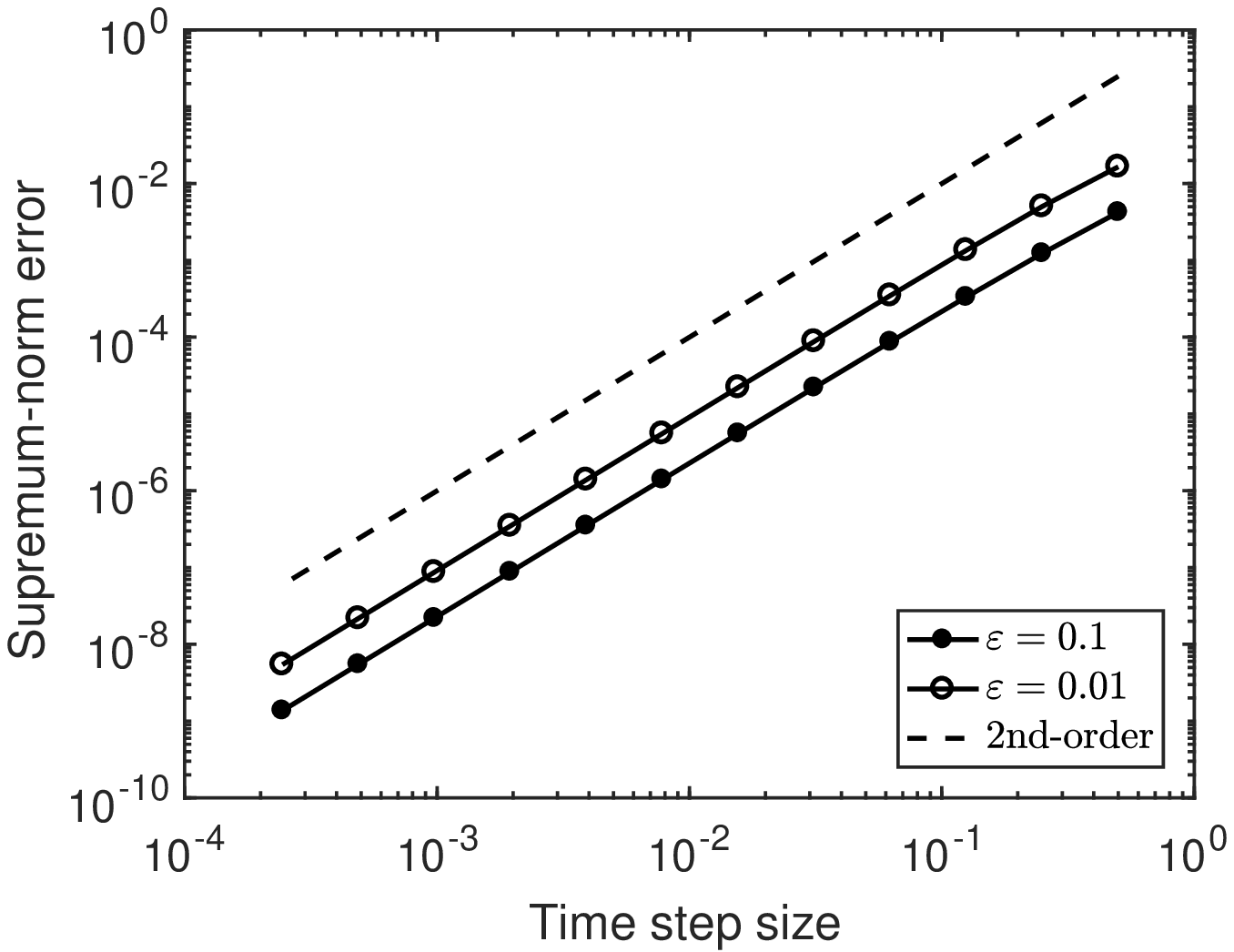}}\\
\subfigure[IFRK3 scheme]{\includegraphics[width=0.5\textwidth]{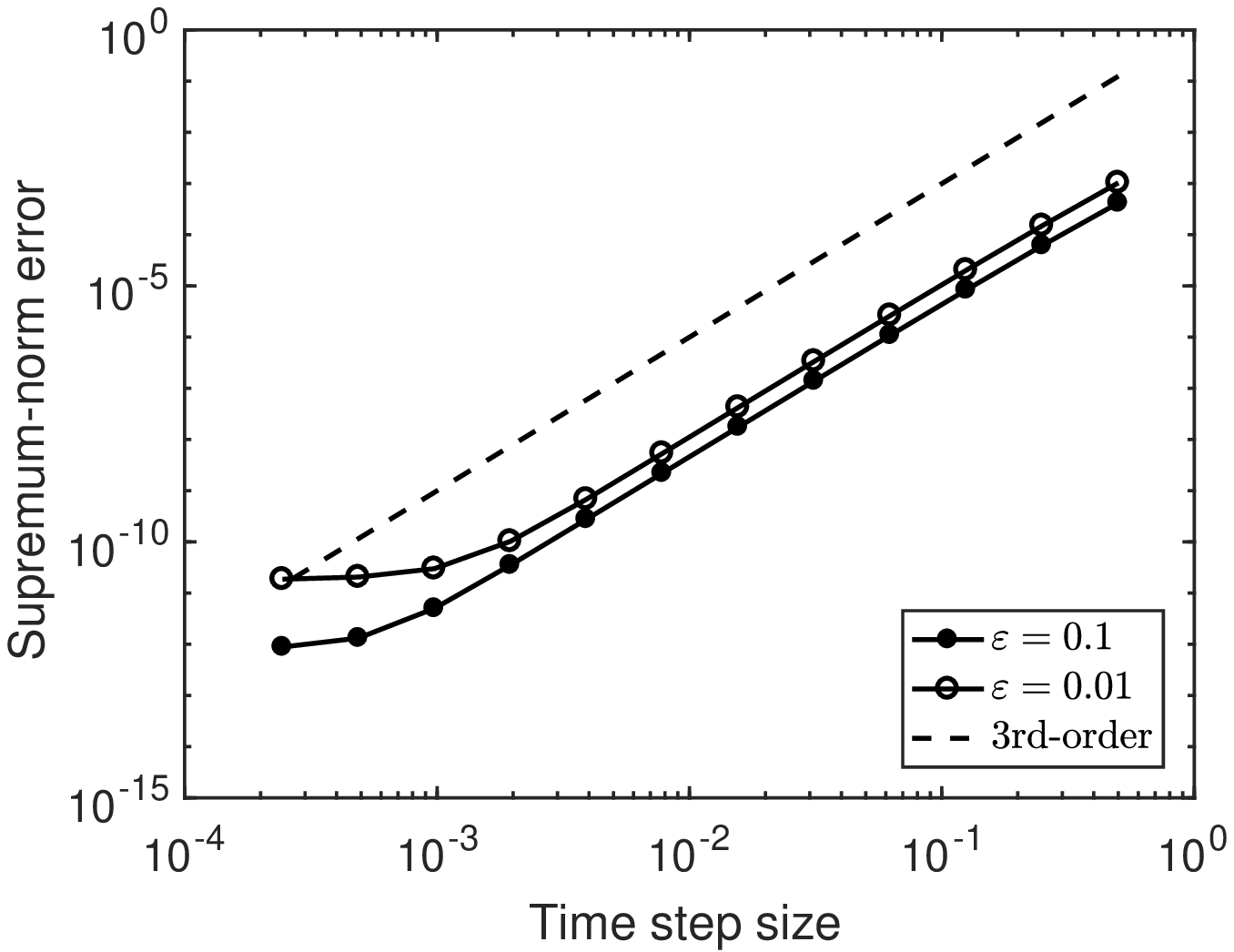}}
\subfigure[IFRK4 scheme]{\includegraphics[width=0.5\textwidth]{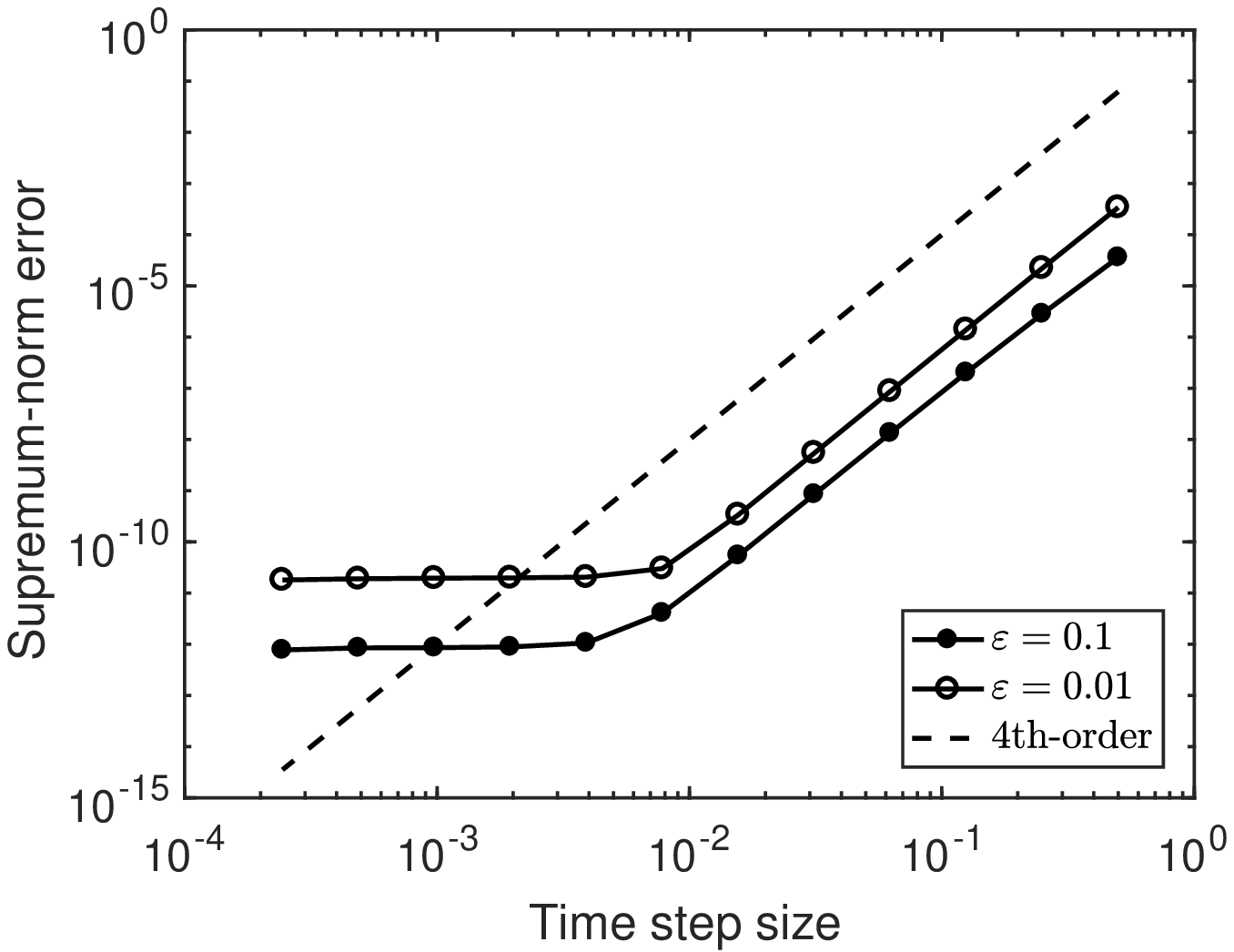}}
\caption{Convergence rates of the IFRK schemes.}
\label{fig_convergence}
\end{figure}

\subsection{Tests for MBP preservation}

Then, we simulate the process of the coarsening dynamics
by setting $\eps=0.1$ and the spatial mesh size $h=1/512$,
where the initial data is given by random numbers on each mesh point ranging from $-0.8$ to $0.8$.

Wef irst use the IFRK schemes \eqref{IFFE}, \eqref{IFRK2}, \eqref{IFRK3}, and \eqref{IFRK4}
with the uniform time step size $\dt=0.08$,
close to the upper bound of the time step sizes determined by \eqref{cond_timestep_IFRK4} for the IFRK4 scheme.
\figurename~\ref{fig_coarsen11} plots the evolutions of the supremum norm of the solution of the four schemes.
The red dash horizonal line shows the theoretical upper bound $\gamma$ of the numerical solutions
and the black solid curve gives a benchmark obtained by using the IFRK4 scheme with $\dt=0.001$.
It can be observed that
the supremum norms of the numerical solutions are always bounded by the theoretical value,
{and more precisely, the values of vertical coordinate of every curve do not exceed $\gamma$,}
which suggests the preservation of the MBP.
In addition, the curves corresponding to the IF1 and IFRK2 schemes
produce obvious deviations from the benchmark due to the low accuracy,
while there is little difference between the curves for the IFRK3 and IFRK4 schemes and that for the benchmark.
This  shows the convergence of the four IFRK schemes and the benefit of high-order accurate schemes.

\begin{figure}[h]
\centering
\includegraphics[width=0.5\textwidth]{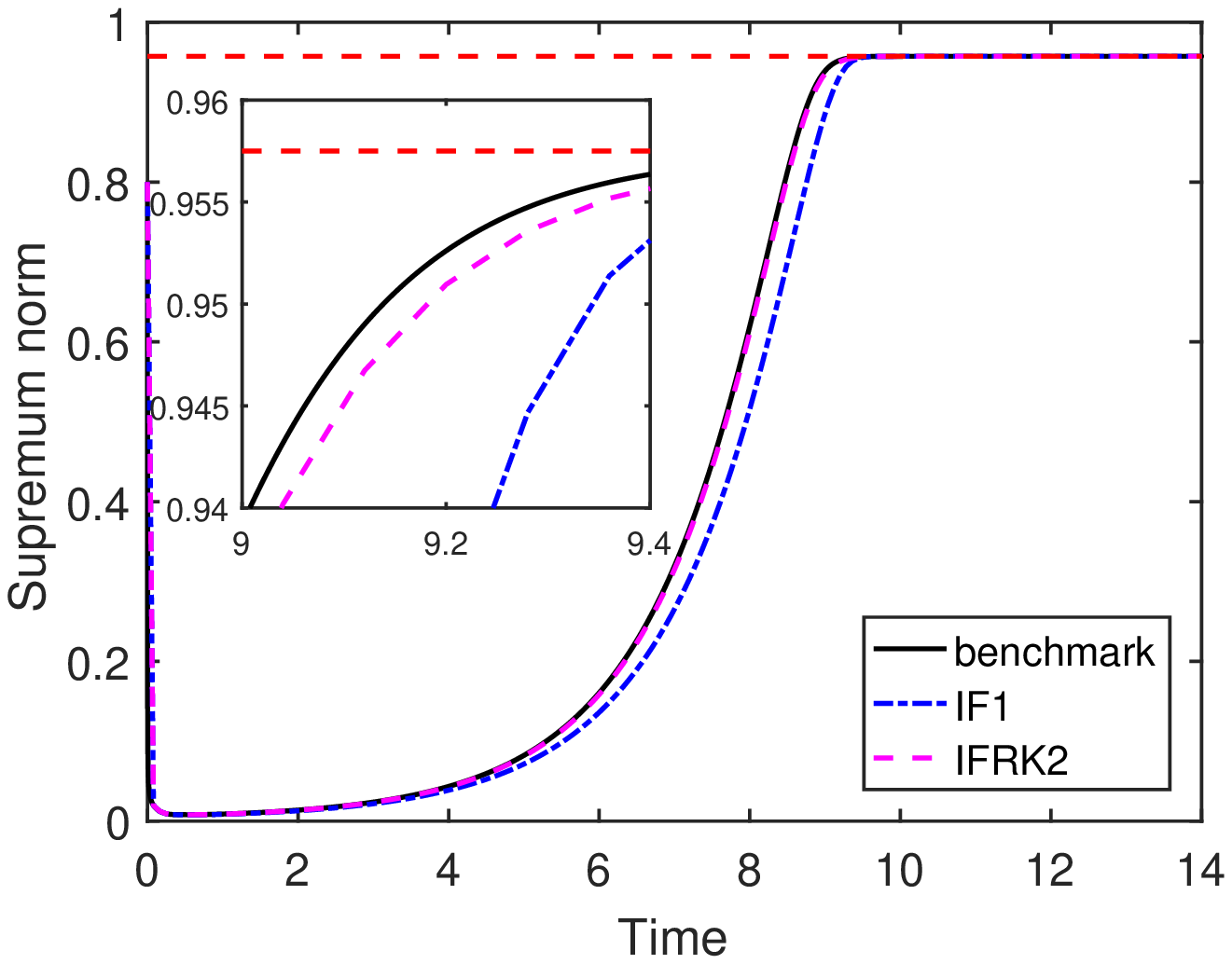}
\includegraphics[width=0.5\textwidth]{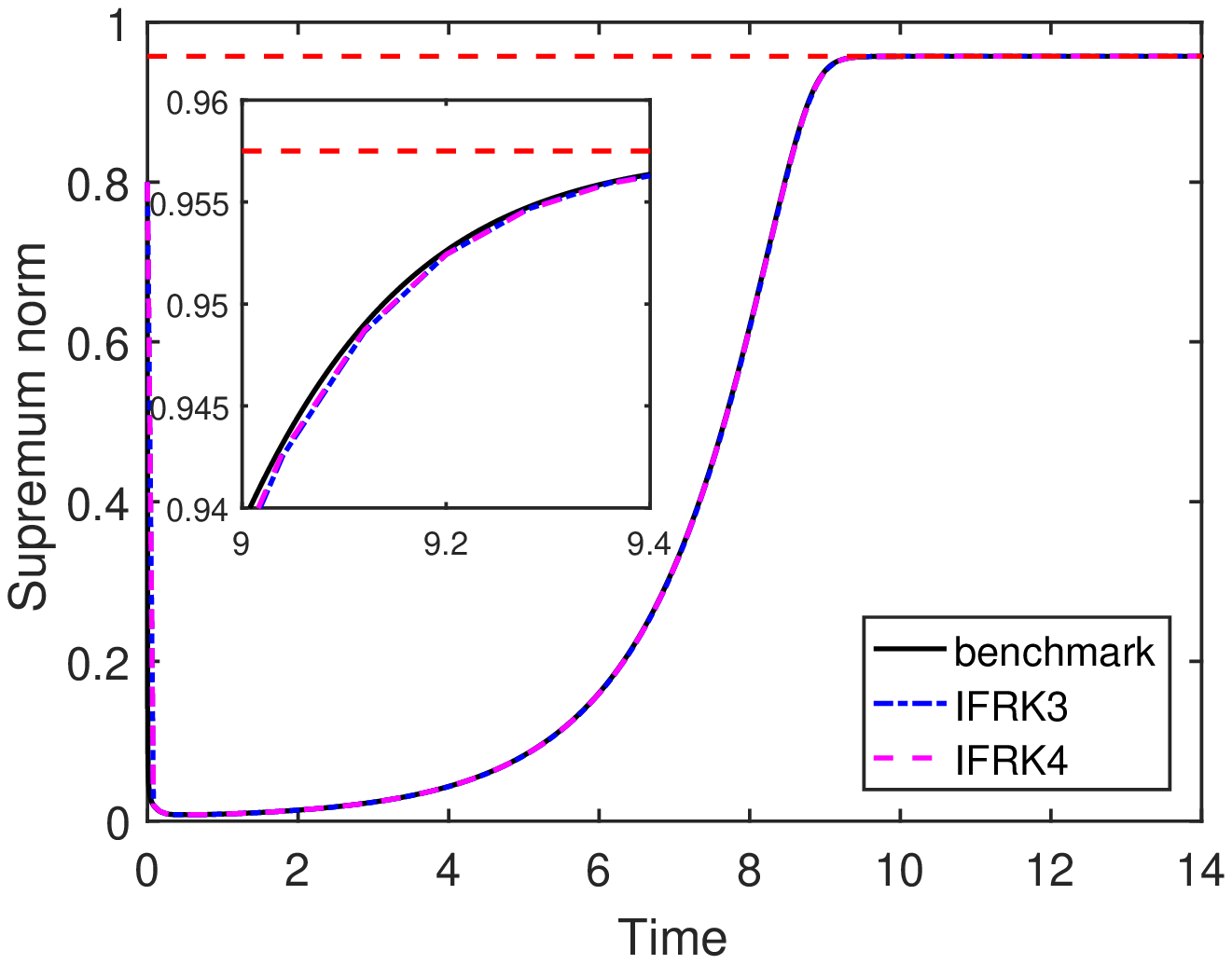}
\caption{Evolutions of the supremum norms of the solutions of IFRK schemes.}
\label{fig_coarsen11}
\end{figure}

As we mentioned in Remark \ref{rmk_ShuOsher},
the third-order IF Shu--Osher scheme \eqref{IFRK3n} may not preserve the maximum bound principle
due to the existence of the negative abscissas in the matrix exponential.
Here, we simulate the coarsening dynamics described above
by using the scheme \eqref{IFRK3n} with a small time step size $\dt=0.005$
to explore the behavior of the numerical solution.
The left graph in \figurename~\ref{fig_coarsen12} shows the evolution of the supremum norm till $t=5$.
We can see the supremum norm of the numerical solution (the solid line)
exceeds the theoretical upper bound (the dash line) around $t=4.8$
and even evolves larger than $1$ after $t=4.9$.
Note that there is a logarithmic term in the nonlinear part
and it will be evaluated by complex numbers if $u$ ranges out of the interval $(-1,1)$.
The right grap in \figurename~\ref{fig_coarsen12} shows that
the energy \eqref{eg_energy} decreases along the time
until a wrong sharp corner arises around $t=4.9$.
We see that the simulation will give a completely wrong result
even though a small time step size is adopted for the scheme \eqref{IFRK3n},
which suggests the necessity of the property of non-decreasing abscissas.
Actually, we also repeat the experiment by using a smaller time step size $0.004$
and obtain the correct result similar to that shown in \figurename~\ref{fig_coarsen11}.
Thus we guess the scheme \eqref{IFRK3n} with the time step size $\dt\le0.004$
could give the correct numerical solutions.
Moreover, according to the second inequality of \eqref{eg_Lexp},
we further guess the scheme \eqref{IFRK3n} may preserve the MBP
when $\dt\le Ch^2$ for some constant $C$,
though we do not have the theoretical proof.

\begin{figure}[h]
\centering
\includegraphics[width=0.5\textwidth]{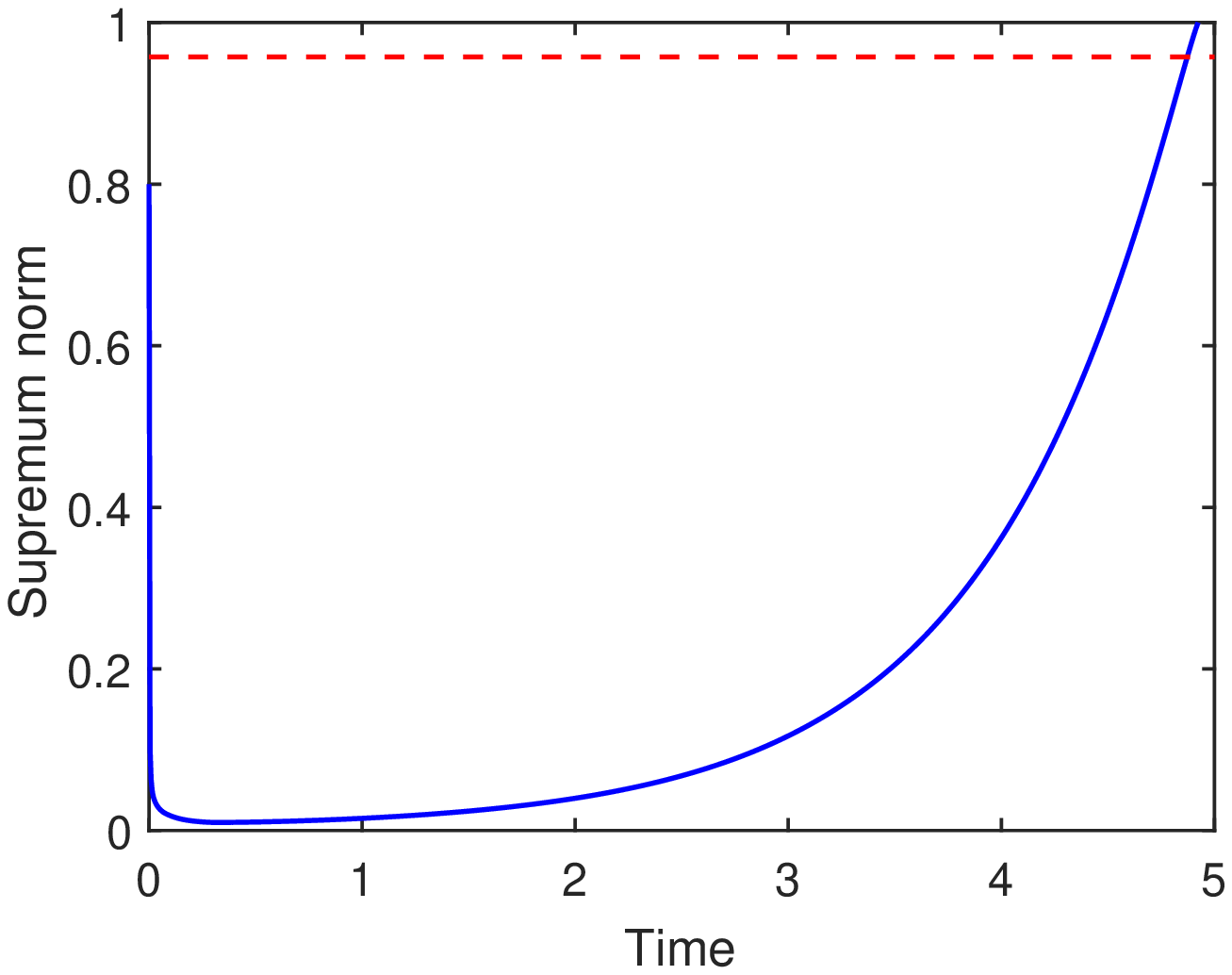}
\includegraphics[width=0.5\textwidth]{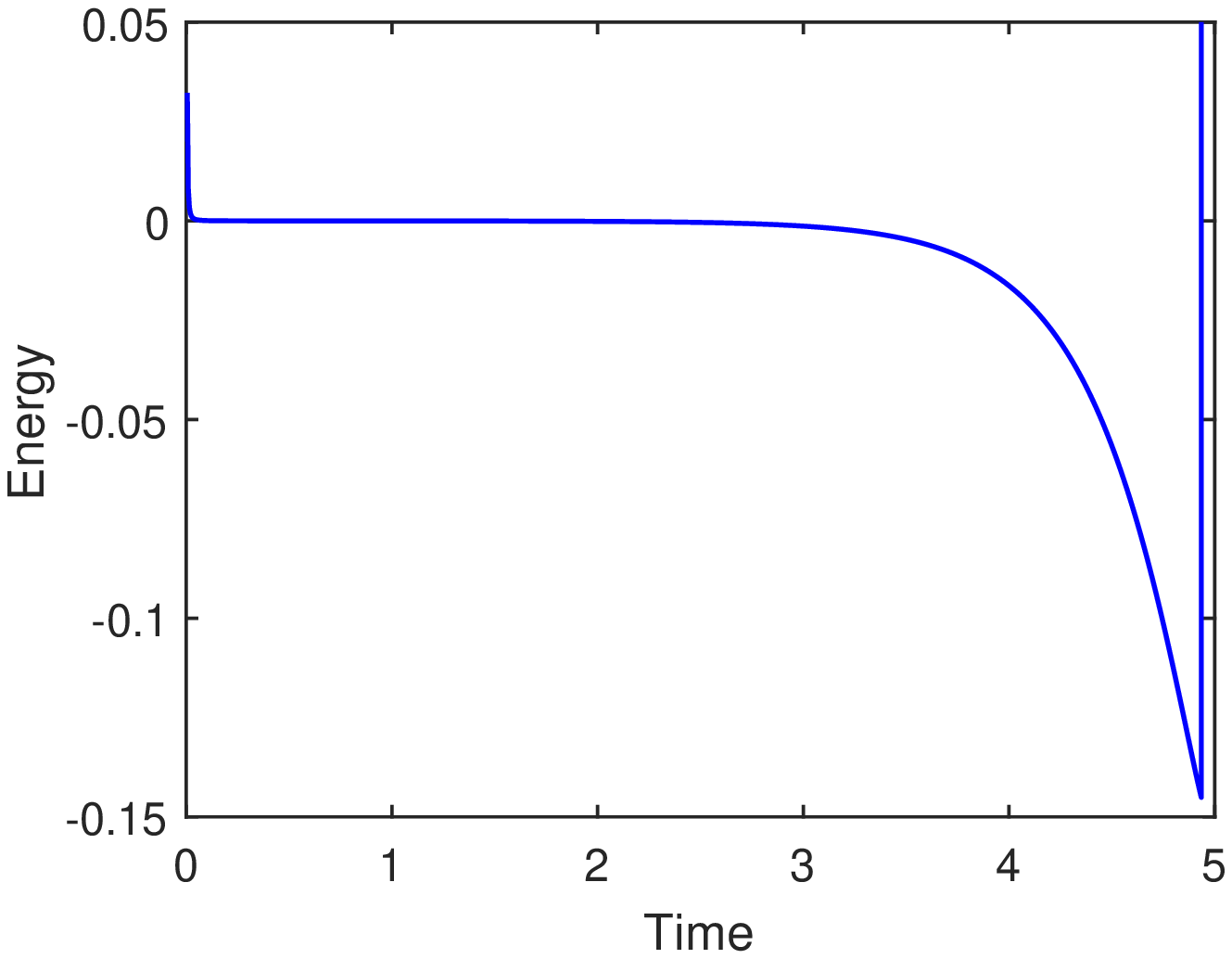}
\caption{Evolutions of the supremum norm (left) and the energy (right)
of the solution of the third-order IF Shu--Osher scheme \eqref{IFRK3n}.}
\label{fig_coarsen12}
\end{figure}

\subsection{Efficiency comparison for long-time simulations}

One may conclude from \figurename~\ref{fig_convergence}-(b) and (d) that
the numerical error of the IFRK2 scheme with the time step size $0.001$
has the same magnitude as the error of the IFRK4 scheme with the time step size about $0.08$.
Thus, we repeat the simulation of the above coarsening dynamics with $\eps=0.01$
by adopting the IFRK2 scheme \eqref{IFRK2} with $\tau=0.001$ and the IFRK4 scheme \eqref{IFRK4} with $\tau=0.08$
to compare the efficiencies of these two schemes with the accuracy at the same level.
The terminal time of the simulation is set to be $T=610$.

The computations are carried out in MATLAB
on a Laptop with a four-core Intel 2.70~GHz Processor and 8~GB Memory.
The CPU time for the computation by the IFRK2 scheme is about $352.28$~minutes
and that for the IFRK4 scheme is around $13.96$~minutes,
approximately $3.96\%$ of the former.
This implies the higher efficiency of the IFRK4 scheme than that of the IFRK2 one.
The numerical results of the IFRK4 scheme are shown in the following pictures (the results by the IFRK2 scheme are almost identical).
\figurename~\ref{fig_coarsen21} shows the configurations of the solution at $t=4,6,10,30,100$, and $300$.
The simulated dynamics begins with a random state
and towards the homogeneous steady state of constant $-\gamma$,
which is reached after about $t=600$ in our simulation.
The evolutions of the supremum norm and the energy are plotted in \figurename~\ref{fig_coarsen22}.
We observe that the energy decreases monotonically as expected
and the MBP is perfectly preserved
so that the solution is always located in the interval $[-\gamma,\gamma]$.

\begin{figure}[h]
\centerline{
\hspace{-0.4cm}
\includegraphics[width=0.37\textwidth]{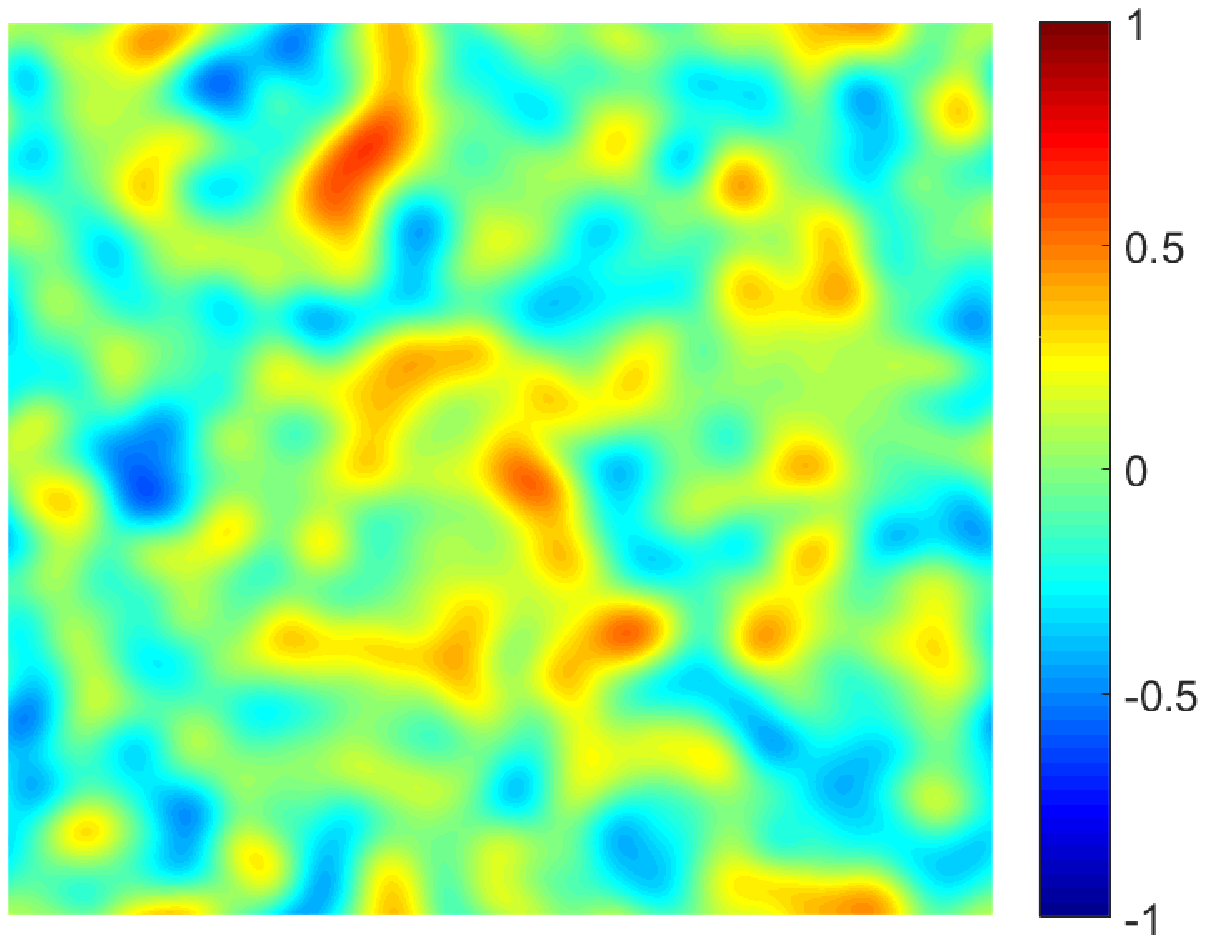}\hspace{-0.5cm}
\includegraphics[width=0.37\textwidth]{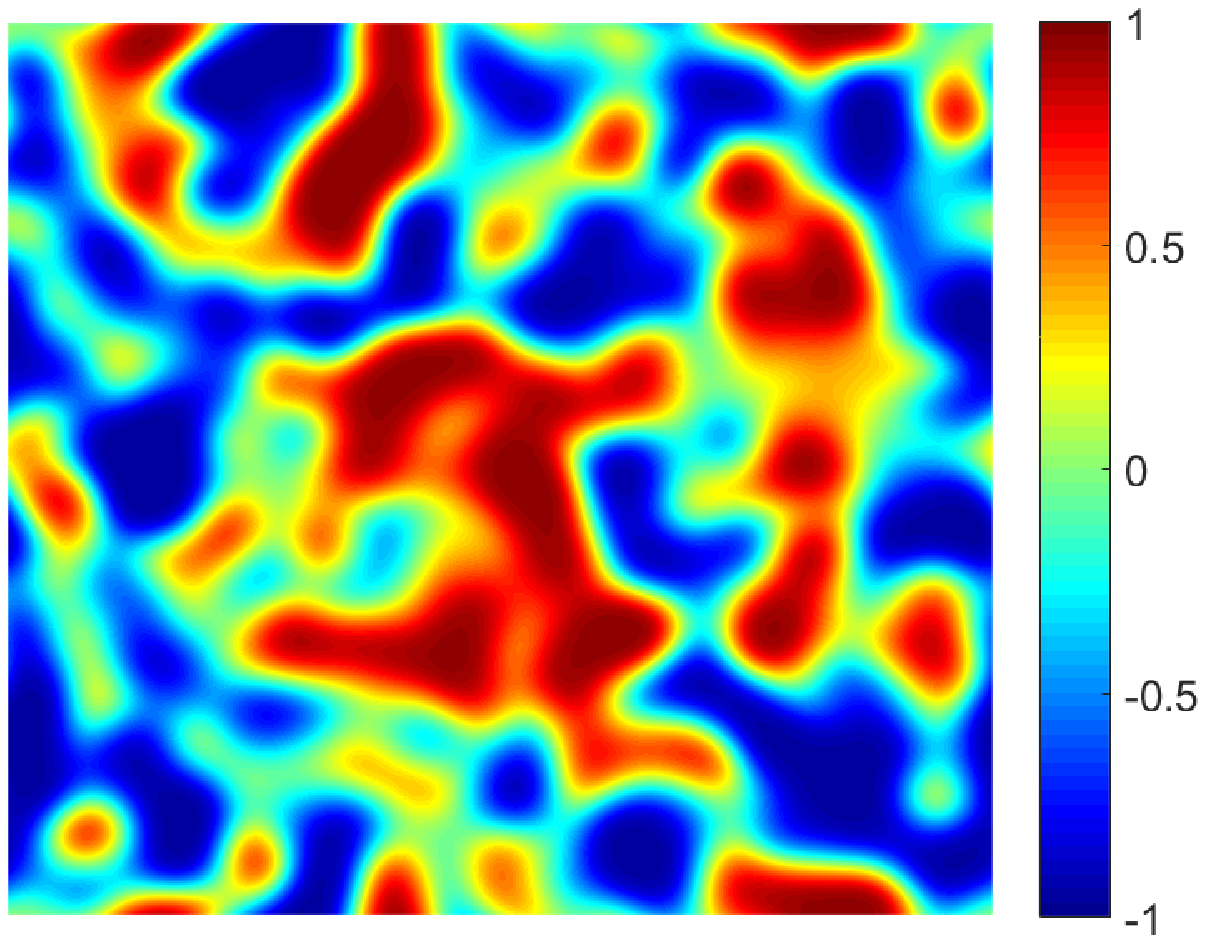}\hspace{-0.5cm}
\includegraphics[width=0.37\textwidth]{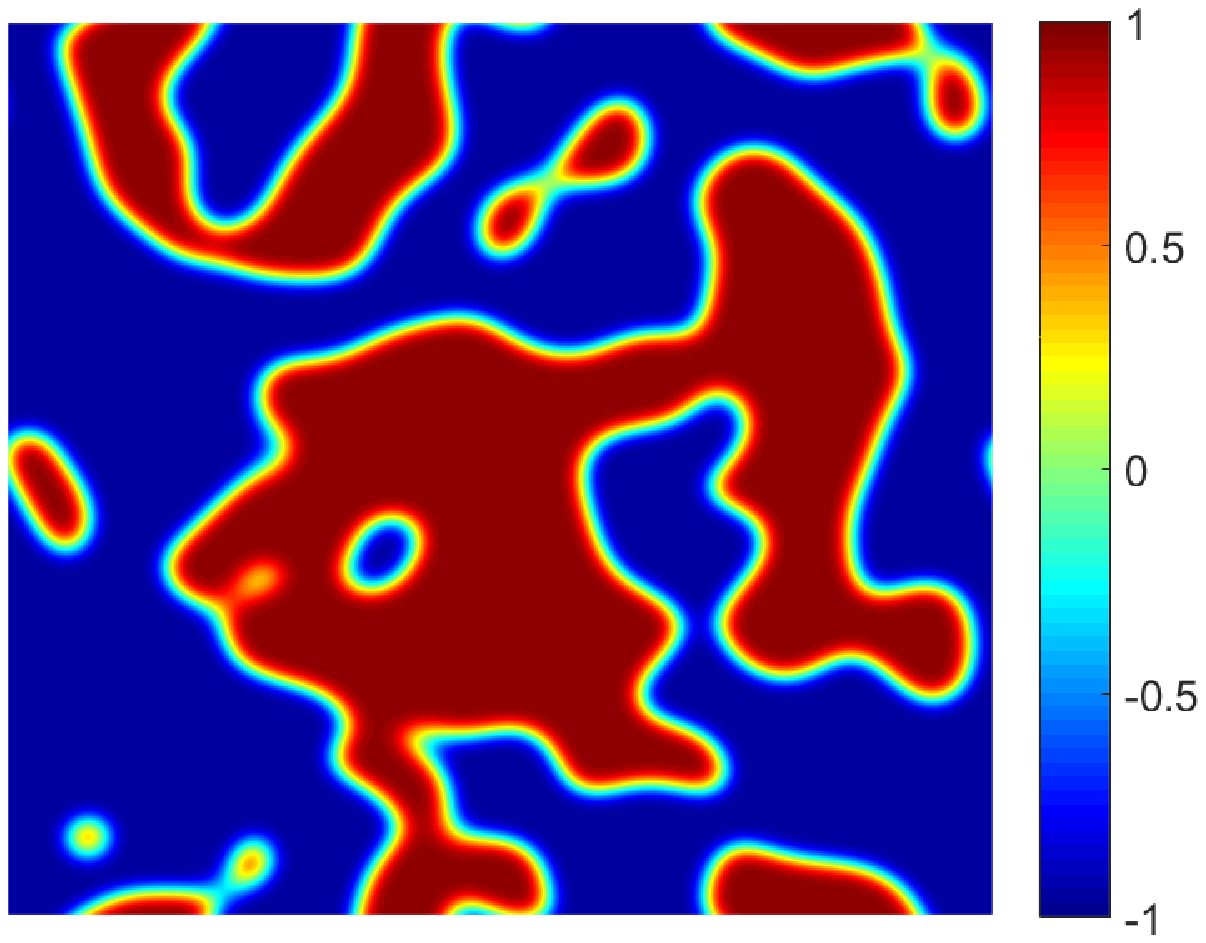}}
\centerline{\hspace{-0.4cm}
\includegraphics[width=0.37\textwidth]{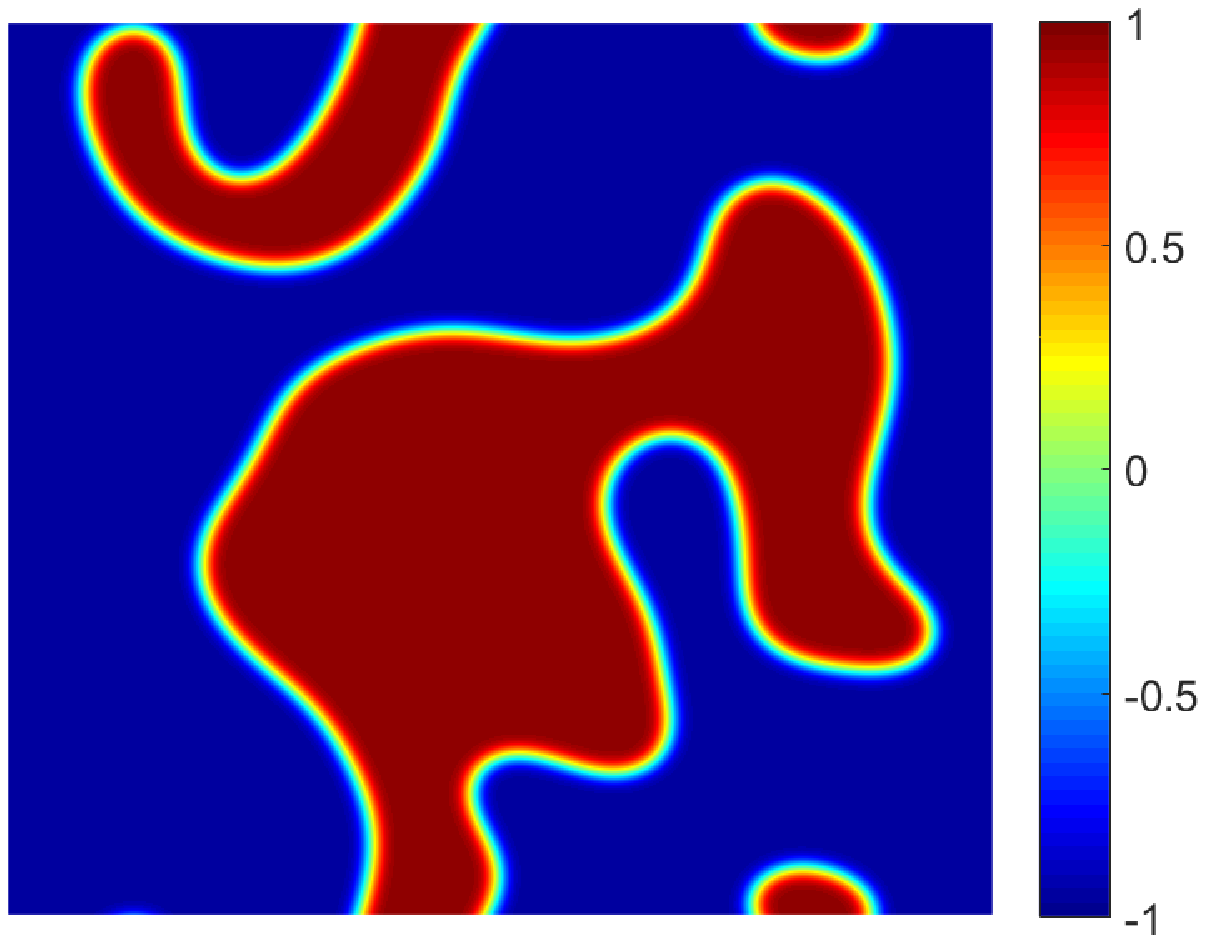}\hspace{-0.5cm}
\includegraphics[width=0.37\textwidth]{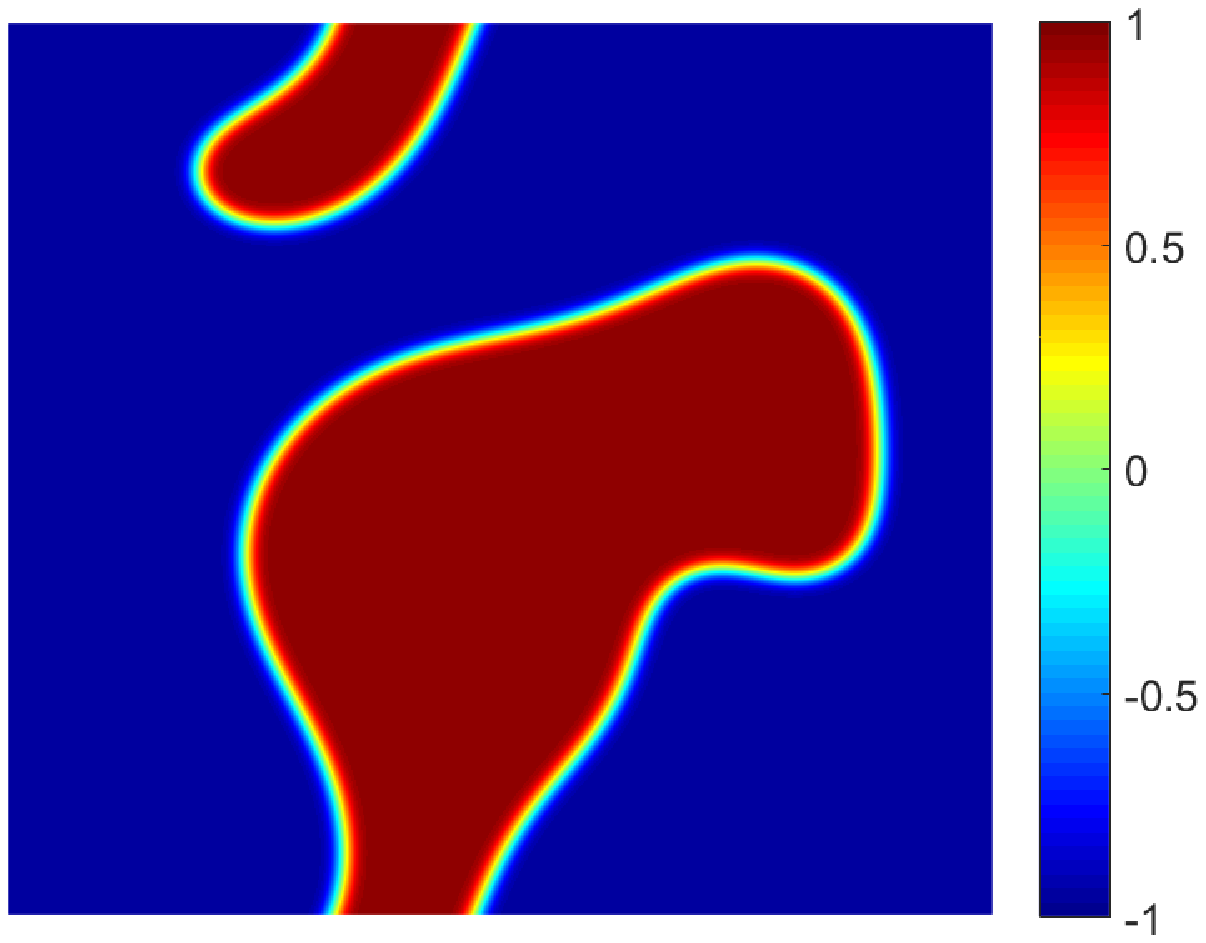}\hspace{-0.5cm}
\includegraphics[width=0.37\textwidth]{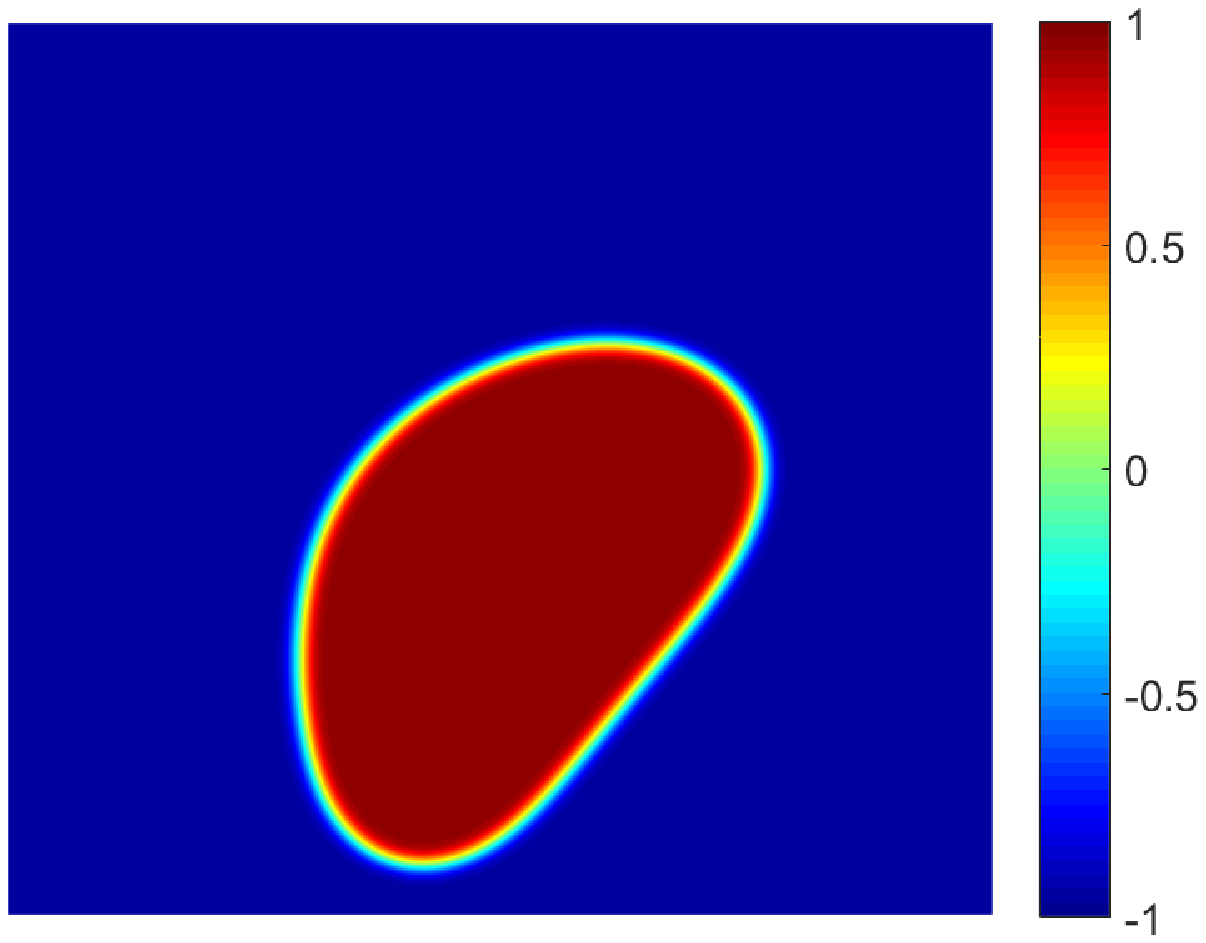}}
\caption{The snapshots of the evolution by the IFRK4 scheme at $t=4,6,10,30,100,300$, respectively
(left to right and top to bottom).}
\label{fig_coarsen21}
\end{figure}

\begin{figure}[h]
\centering
\includegraphics[width=0.5\textwidth]{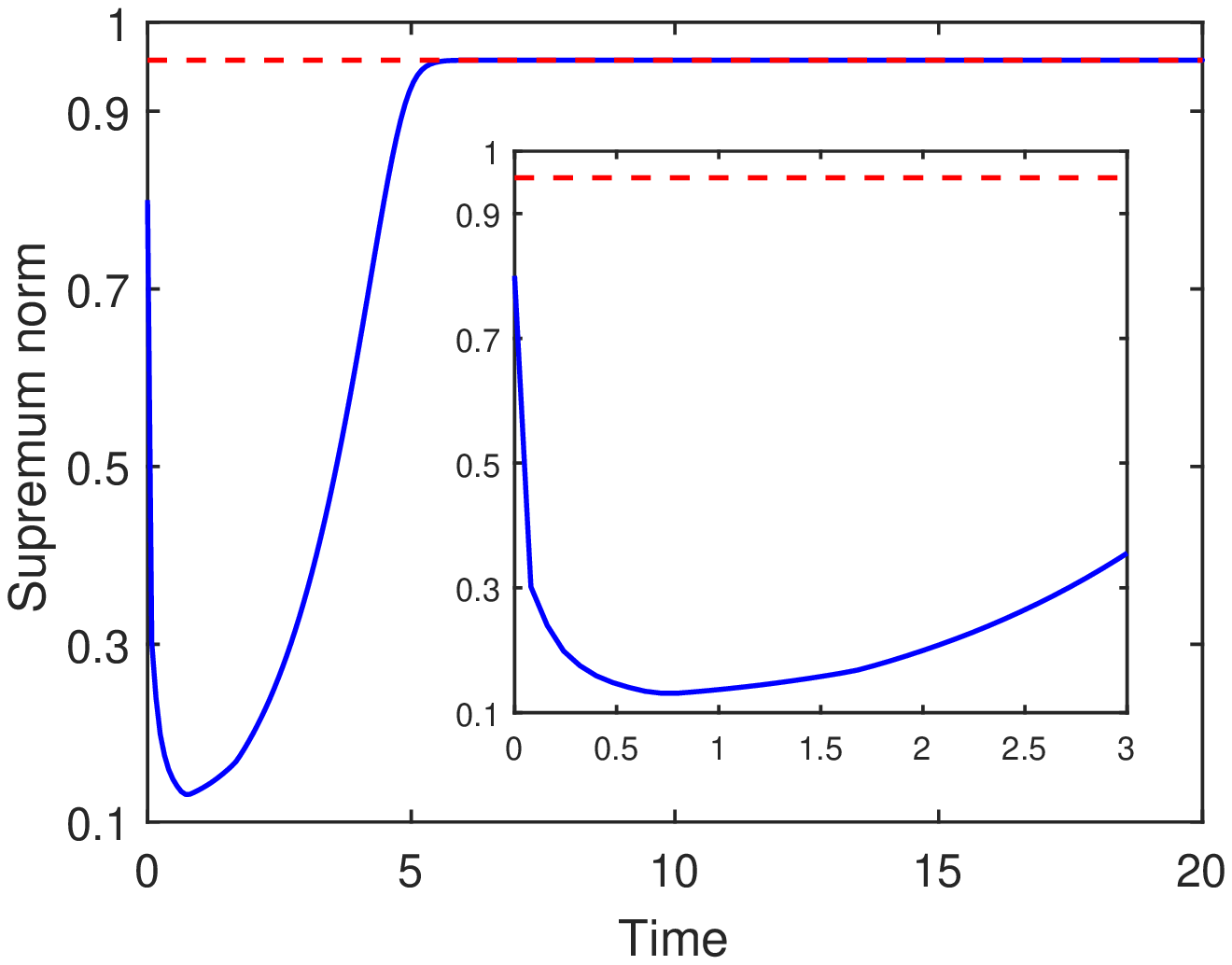}
\includegraphics[width=0.5\textwidth]{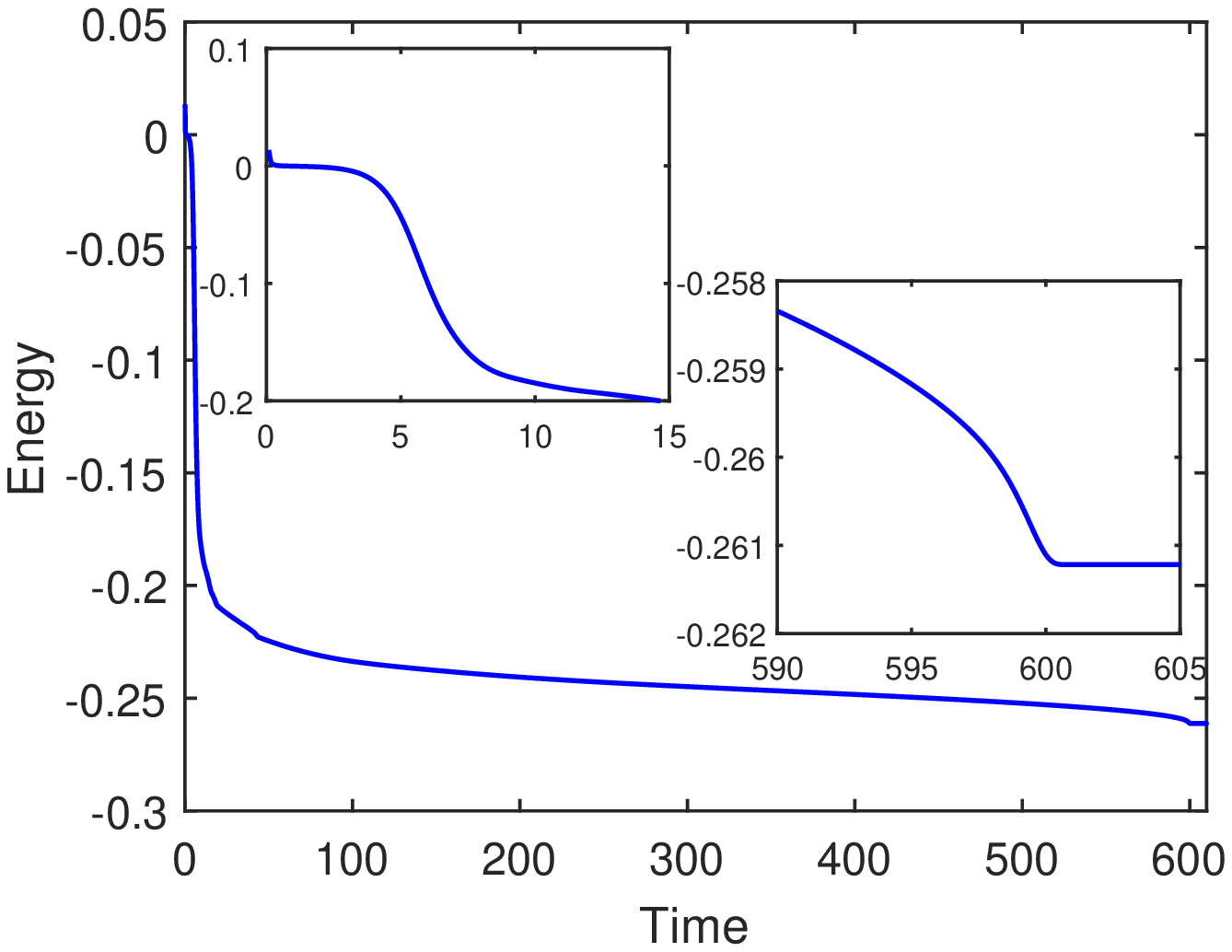}
\caption{Evolutions of the supremum norm (left) and the energy (right) by the IFRK4 scheme.}
\label{fig_coarsen22}
\end{figure}

\section{Concluding remarks}\label{sect_con}

In this work, we study the fully-discrete MBP-preserving IFRK method
for semilinear parabolic equations  taking the form \eqref{model_pde}.
We show that the IFRK method in general form preserves the MBP under certain conditions,
present several practical and specific IFRK schemes up to fourth order,
and give the convergence analysis of the method.
As shown in Theorem \ref{thm_MBP},
the constraints on the time step size do not depend on the linear part,
which is a significant advantage in comparison with the standard explicit SSP-RK methods where a CFL restriction is often needed.
The four-stage IFRK4 scheme \eqref{IFRK4} provides
a high-order MBP-preserving numerical scheme for the first time
and its high efficiency was verified by the numerical simulation of long-time evolutions.

There are some related problems worthy to explore further as continuation of this paper.
On one hand, the IF1, IFRK2, and IFRK3 schemes presented in Section \ref{sect_IFRKschemes}
actually come from the SSP-RK schemes in \cite{GoShTa01} for the system \eqref{model_eq} with $L=0$,
while the IFRK4 scheme comes from the classic fourth-order RK method with the inevitable negative $\beta_{ij}$.
Thanks to the conditions \eqref{cond_f} and \eqref{cond_f2} for the nonlinear function,
the nonnegativity constraint of $\beta_{ij}$ is not necessary in our framework.
Therefore, one of our future works is to explore
whether it is possible to find a fifth-order (or even higher-order) MBP-preserving IFRK scheme
by using similar routine as finding SSP schemes
without the requirement of the nonnegativity of $\beta_{ij}$.
On the other hand, the preservation of the MBP requires a constraint on the time step size
due to the conditions \eqref{cond_f} and \eqref{cond_f2}.
Thus, it is also expected to answer
whether one can use the stabilizing technique,
by adding a stabilization term as done in \cite{DuJuLiQi19,ShTaYa16},
to remove the requirement on the time step size.
In addition, the generalization to vector- and matrix-valued MBPs,
including the complex Ginzburg--Landau model \cite{DuGuPe92}
and orthogonal matrix-valued equations \cite{OsWa20} as the examples,
could be also considered as done in \cite{DuJuLiQi20review}.

\section*{Acknowledgments}

We are grateful to Professor Chi-Wang Shu of Brown University for many valuable comments.
This work is supported by the CAS AMSS-PolyU Joint Laboratory of Applied Mathematics.
L. Ju's work is partially supported by US National Science Foundation grant DMS-1818438
and US Department of Energy grant DE-SC0020270.
X. Li's work is partially supported by National Natural Science Foundation of China grant 11801024.
Z. Qiao's work is partially supported by the Hong Kong Research Council GRF grants 15300417 and 15302919
and the Hong Kong Polytechnic University fund G-UAEY.
J. Yang's work is supported by National Natural Science Foundation of China grant 11871264,
Natural Science Foundation of Guangdong Province (2018A0303130123),
and NSFC/Hong Kong RGC Joint Research Scheme (NFSC/RGC 11961160718).

\section*{References}

\bibliographystyle{elsarticle-num}

\begin{thebibliography}{0}

\bibitem{AhLi19}
S. Ahmed and X.~F. Liu,
{\it High order integration factor methods for systems with inhomogeneous boundary conditions},
J. Comput. Appl. Math., 348 (2019), 89--102.

\bibitem{AlCa79}
S.~M. Allen and J.~W. Cahn,
{\it A microscopic theory for antiphase boundary motion and its application to antiphase domain coarsening},
Acta Metall., 27 (1979), 1085--1095.


\bibitem{Dahlquist58}
G. Dahlquist,
{\it Stability and Error Bounds in the Numerical Integration of Ordinary Differential Equations},
Trans. Royal Inst. of Technology, No.~130, Stockholm, 1959.

\bibitem{Du98}
Q. Du,
{\it Discrete gauge invariant approximations of a time-dependent Ginzburg--Landau model of superconductivity},
Math. Comp., 67 (1998), 965--986.

\bibitem{Du05}
Q. Du,
{\it Numerical approximations of the Ginzburg--Landau models for superconductivity},
J. Math. Phys., 46 (2005), 095109.

\bibitem{DuGuLeZh12}
Q. Du, M. Gunzburger, R.~B. Lehoucq, and K. Zhou,
{\it Analysis and approximation of nonlocal diffusion problems with volume constraints},
SIAM Rev., 54 (2012), 667--696.

\bibitem{DuGuPe92}
Q. Du, M. Gunzburger, and J. Peterson,
{\it Analysis and approximation of Ginzburg--Landau models for superconductivity},
SIAM Rev., 34 (1992), 54--81.

\bibitem{DuJuLiQi18}
Q. Du, L. Ju, X. Li, and Z.~H. Qiao,
{\it Stabilized linear semi-implicit schemes for the nonlocal Cahn--Hilliard equation},
J. Comput. Phys., 363 (2018), 39--54.

\bibitem{DuJuLiQi19}
Q. Du, L. Ju, X. Li, and Z.~H. Qiao,
{\it Maximum principle preserving exponential time differencing schemes for the nonlocal Allen--Cahn equation},
SIAM J. Numer. Anal., 57 (2019), 875--898.

\bibitem{DuJuLiQi20review}
Q. Du, L. Ju, X. Li, and Z.~H. Qiao,
{\it Maximum bound principles for a class of semilinear parabolic equations and exponential time differencing schemes},
SIAM Rev., accepted, 2020.

\bibitem{DuJuLu19}
Q. Du, L. Ju, and J.~F. Lu,
{\it Analysis of fully discrete approximations for dissipative systems and application to time-dependent nonlocal diffusion problems},
J. Sci. Comput., 78 (2019), 1438--1466.

\bibitem{DuTaTiYa19}
Q. Du, Y.~Z. Tao, X.~C. Tian, and J. Yang,
{\it Asymptotically compatible discretization of multidimensional nonlocal diffusion models
and approximation of nonlocal Green's functions},
IMA J. Numer. Anal., 39 (2019), 607--625.

\bibitem{DuYaZh19}
Q. Du, J. Yang, and Z. Zhou,
{\it Time-fractional Allen--Cahn equations: analysis and numerical methods},
arXiv:1906.06584, 2019.

\bibitem{DuZh05}
Q. Du and W.-X. Zhu,
{\it Analysis and applications of the exponential time differencing schemes and their contour integration modifications},
BIT Numer. Math., 45 (2005), 307--328.

\bibitem{EvSoSo92}
L.~C. Evans, H.~M. Soner, and P.~E. Souganidis,
{\it Phase transitions and generalized motion by mean curvature},
Commun. Pure Appl. Math., 45 (1992), 1097--1123.

\bibitem{FeTaYa13}
X.~L. Feng, T. Tang, and J. Yang,
{\it Stabilized Crank--Nicolson/Adams--Bashforth schemes for phase field models},
East Asian J. Appl. Math., 3 (2013), 59--80.

\bibitem{GaJuXi19}
H.~D. Gao, L. Ju, and W. Xie,
{\it A stabilized semi-implicit Euler gauge-invariant method for the time-dependent Ginzburg--Landau equations},
J. Sci. Comput., 80 (2019), 1083--1115.

\bibitem{GoSh98}
S. Gottlieb and C.-W. Shu,
{\it Total variation diminishing Runge--Kutta schemes},
Math. Comp., 67 (1998), 73--85.

\bibitem{GoShTa01}
S. Gottlieb, C.-W. Shu, and E. Tadmor,
{\it Strong stability-preserving high-order time discretization methods},
SIAM Rev., 43 (2001), 89--112.

\bibitem{GuWaWi14}
Z. Guan, C. Wang, and S.~M. Wise,
{\it A convergent convex splitting scheme for the periodic nonlocal Cahn--Hilliard equation},
Numer. Math., 128 (2014), 377--406.

\bibitem{HairerNoWa93}
E. Hairer, S.~P. Norsett, and G. Wanner,
{\it Solving Ordinary Differential Equations I: Nonstiff Problems} (2nd ed.),
Springer Ser. Comput. Math., 8, Springer-Verlag, Berlin, 1993.

\bibitem{HoLe20}
T.~L. Hou and H.~T. Leng,
{\it Numerical analysis of a stabilized Crank--Nicolson/Adams--Bashforth finite difference scheme for Allen--Cahn equations},
Appl. Math. Lett., 102 (2020), 106150.

\bibitem{HoTaYa17}
T.~L. Hou, T. Tang, and J. Yang,
{\it Numerical analysis of fully discretized Crank--Nicolson scheme for fractional-in-space Allen--Cahn equations},
J. Sci. Comput., 72 (2017), 1214--1231.

\bibitem{HoXiJi20}
T.~L. Hou, D.~F. Xiu, and W.~Z. Jiang,
{\it A new second-order maximum-principle preserving finite difference scheme
for Allen--Cahn equations with periodic boundary conditions},
Appl. Math. Lett., 104 (2020), 106265.

\bibitem{IsGrGo18}
L. Isherwood, Z.~J. Grant, and S. Gottlieb,
{\it Strong stability preserving integrating factor Runge--Kutta methods},
SIAM J. Numer. Anal., 56 (2018), 3276--3307.

\bibitem{JuLiQiZh18}
L. Ju, X. Li, Z.~H. Qiao, and H. Zhang,
{\it Energy stability and error estimates of exponential time differencing schemes
for the epitaxial growth model without slope selection},
Math. Comp., 87 (2018), 1859--1885.

\bibitem{JuLiLe14}
L. Ju, X.~F. Liu, and W. Leng,
{\it Compact implicit integration factor methods for a family of semilinear fourth-order parabolic equations},
Discrete Contin. Dyn. Syst. Ser. B, 19 (2014), 1667--1687.

\bibitem{NePaVa12}
E. Nezza, G. Palatucci, and E. Valdinoci,
{\it Hitchhiker's guide to the fractional Sobolev spaces},
Bull. Sci. Math., 136 (2012), 521--573.

\bibitem{OsWa20}
B. Osting and D. Wang,
{\it A diffusion generated method for orthogonal matrix-valued fields},
Math. Comp., 89 (2020), 515--550.

\bibitem{PeRo76}
D.~Y. Peng and D.~B. Robinson,
{\it A new two-constant equation of state},
Ind. Eng. Chem. Fund., 15 (1976), 59--64.

\bibitem{QiSu14}
Z.~H. Qiao and S.~Y. Sun,
{\it Two-phase fluid simulation using a diffuse interface model with Peng--Robinson equation of state},
SIAM J. Sci. Comput., 36 (2014), B708--B728.

\bibitem{QiZhTa11}
Z.~H. Qiao, Z.~R. Zhang, and T. Tang,
{\it An adaptive time-stepping strategy for the molecular beam epitaxy models},
SIAM J. Sci. Comput., 33 (2011), 1395--1414.

\bibitem{ShTaYa16}
J. Shen, T. Tang, and J. Yang,
{\it On the maximum principle preserving schemes for the generalized Allen--Cahn equation},
Commun. Math. Sci., 14 (2016), 1517--1534.

\bibitem{ShWaWaWi12}
J. Shen, C. Wang, X.~M. Wang, and S.~M. Wise,
{\it Second-order convex splitting schemes for gradient flows with Ehrlich--Schwoebel type energy:
application to thin film epitaxy},
SIAM J. Numer. Anal., 50 (2012), 105--125.

\bibitem{ShXuYa19}
J. Shen, J. Xu, and J. Yang,
{\it A new class of efficient and robust energy stable schemes for gradient flows},
SIAM Rev., 61 (2019), 474--506.

\bibitem{ShYa10b}
J. Shen and X.~F. Yang,
{\it Numerical approximations of Allen--Cahn and Cahn--Hilliard equations},
Discrete Contin. Dyn. Syst., 28 (2010), 1669--1691.

\bibitem{ShOs88}
C.-W. Shu and S. Osher,
{\it Efficient implementation of essentially non-oscillatory shockcapturing schemes},
J. Comput. Phys., 77 (1988), 439--471.

\bibitem{Soderlind06}
G. S\"oderlind,
{\it The logarithmic norm. History and modern theory},
BIT, 46 (2006), 631--652.

\bibitem{StVo15}
P. Stehl\'ik and J. Volek,
{\it Maximum principles for discrete and semidiscrete reaction-diffusion equation},
Discrete Dyn. Nat. Soc., 2015, 791304.

\bibitem{TaWaNi15}
C. Ta, D.~Y. Wang, and Q. Nie,
{\it An integration factor method for stochastic and stiff reaction-diffusion systems},
J. Comput. Phys., 295 (2015), 505--522.

\bibitem{TaYa16}
T. Tang and J. Yang,
{\it Implicit-explicit scheme for the Allen--Cahn equation preserves the maximum principle},
J. Comput. Math., 34 (2016), 471--481.

\bibitem{tangyang19}
T. Tang and J. Yang,
{\it Framework of monotone schemes preserving the maximum principle for Allen--Cahn equations},
preprint, 2019.

\bibitem{TiZhDe15}
W.~Y. Tian, H. Zhou, and W.~H. Deng,
{\it A class of second order difference approximations for solving space fractional diffusion equations},
Math. Comp., 84 (2015), 1703--1727.

\bibitem{WiWaLo09}
S.~M. Wise, C. Wang, and J.~S. Lowengrub,
{\it An energy stable and convergent finite difference scheme for the phase field crystal equation},
SIAM J. Numer. Anal., 47 (2009), 2269--2288.

\bibitem{XiFeYu17}
X.~F. Xiao, X.~L. Feng, and J.~Y. Yuan,
{\it The stabilized semi-implicit finite element method for the surface Allen--Cahn equation},
Discrete Contin. Dyn. Syst. Ser. B, 22 (2017), 2857--2877.

\bibitem{XuTa06}
C.~J. Xu and T. Tang,
{\it Stability analysis of large time-stepping methods for epitaxial growth models},
SIAM J. Numer. Anal., 44 (2006), 1759--1779.

\bibitem{YaDuZh18}
J. Yang, Q. Du, and W. Zhang,
{\it Uniform $L^p$-bound of the Allen--Cahn equation and its numerical discretization},
Int. J. Numer. Anal. Mod., 18 (2018), 213--227.

\bibitem{Yang16}
X.~F. Yang,
{\it Linear, first and second-order, unconditionally energy stable numerical schemes
for the phase field model of homopolymer blends},
J. Comput. Phys., 327 (2016), 294--316.

\bibitem{Ying00}
L.~A. Ying,
{\it A second order explicit finite element scheme to multi-dimensional conservation laws and its convergence},
Sci. China Ser. A, 43 (2000), 945--957.

\end{thebibliography}

\end{document}